\documentclass[12pt]{amsart}

\usepackage{amsxtra,amssymb,amsthm,amsmath,amscd,url,amsrefs}

\usepackage{enumitem}
\newlist{enumth}{enumerate}{1}
\setlist[enumth]{label=\emph{(\arabic*)}, ref=(\arabic*)}

\usepackage[utf8]{inputenc}
\usepackage{eucal}
\usepackage{fullpage}
\usepackage{mathrsfs}
\usepackage{csquotes}
\usepackage[colorlinks]{hyperref}
\usepackage[]{graphicx}
\usepackage{subcaption}
\usepackage{dsfont}
\usepackage{float}
\renewcommand{\mathcal}{\mathscr}

\def\setminus{\mathchoice
	{\mathbin{\vrule height .62ex width 1.61ex depth -.38ex}}
	{\mathbin{\vrule height .62ex width 1.61ex depth -.38ex}}
	{\mathbin{\vrule height .50ex width 0.85ex depth -.28ex}}
	{\mathbin{\vrule height .20ex width 0.570ex depth -.24ex}}
}

\DeclareMathSymbol{A}{\mathalpha}{operators}{`A}%
\DeclareMathSymbol{B}{\mathalpha}{operators}{`B}%
\DeclareMathSymbol{C}{\mathalpha}{operators}{`C}%
\DeclareMathSymbol{D}{\mathalpha}{operators}{`D}%
\DeclareMathSymbol{E}{\mathalpha}{operators}{`E}%
\DeclareMathSymbol{F}{\mathalpha}{operators}{`F}%
\DeclareMathSymbol{G}{\mathalpha}{operators}{`G}%
\DeclareMathSymbol{H}{\mathalpha}{operators}{`H}%
\DeclareMathSymbol{I}{\mathalpha}{operators}{`I}%
\DeclareMathSymbol{J}{\mathalpha}{operators}{`J}%
\DeclareMathSymbol{K}{\mathalpha}{operators}{`K}%
\DeclareMathSymbol{L}{\mathalpha}{operators}{`L}%
\DeclareMathSymbol{M}{\mathalpha}{operators}{`M}%
\DeclareMathSymbol{N}{\mathalpha}{operators}{`N}%
\DeclareMathSymbol{O}{\mathalpha}{operators}{`O}%
\DeclareMathSymbol{P}{\mathalpha}{operators}{`P}%
\DeclareMathSymbol{Q}{\mathalpha}{operators}{`Q}%
\DeclareMathSymbol{R}{\mathalpha}{operators}{`R}%
\DeclareMathSymbol{S}{\mathalpha}{operators}{`S}%
\DeclareMathSymbol{T}{\mathalpha}{operators}{`T}%
\DeclareMathSymbol{U}{\mathalpha}{operators}{`U}%
\DeclareMathSymbol{V}{\mathalpha}{operators}{`V}%
\DeclareMathSymbol{W}{\mathalpha}{operators}{`W}%
\DeclareMathSymbol{X}{\mathalpha}{operators}{`X}%
\DeclareMathSymbol{Y}{\mathalpha}{operators}{`Y}%
\DeclareMathSymbol{Z}{\mathalpha}{operators}{`Z}%

\renewcommand{\leq}{\leqslant}
\renewcommand{\geq}{\geqslant}

\newcommand{\uple}[1]{\text{\boldmath${#1}$}}

\newcommand{\Cc}{\mathbf{C}}
\newcommand{\Oo}{\mathbf{O}}

\newcommand{\Zz}{\mathbf{Z}}

\newcommand{\Rr}{\mathbf{R}}
\newcommand{\Gg}{\mathbf{G}}
\newcommand{\Ss}{\mathbf{S}}

\newcommand{\Qq}{\mathbf{Q}}

\newcommand{\Ff}{\mathbf{F}}

\newcommand{\mmu}{\boldsymbol{\mu}}
\newcommand{\expect}{\mathbf{E}}

\newcommand{\mods}[1]{\,(\mathrm{mod}\,{#1})}

\DeclareMathOperator{\Kl}{Kl}

\renewcommand{\hat}{\widehat}

\DeclareMathOperator{\diam}{diam}
\DeclareMathOperator{\wass}{\mathcal{W}}

\DeclareMathOperator{\Tr}{tr}

\DeclareMathOperator{\End}{End}

\newcommand{\eps}{\varepsilon}
\renewcommand{\rho}{\varrho}

\DeclareMathOperator{\SL}{SL}

\DeclareMathOperator{\Sp}{Sp}

\DeclareMathOperator{\SU}{SU}
\DeclareMathOperator{\Un}{U}
\DeclareMathOperator{\USp}{USp}

\DeclareMathSymbol{\gena}{\mathord}{letters}{"3C}
\DeclareMathSymbol{\genb}{\mathord}{letters}{"3E}

\theoremstyle{plain}
\newtheorem{theorem}{Theorem}[section]
\newtheorem{lemma}[theorem]{Lemma}
\newtheorem{corollary}[theorem]{Corollary}

\newtheorem{proposition}[theorem]{Proposition}

\theoremstyle{remark}

\newtheorem*{rem}{Remark}

\theoremstyle{definition}

\newtheorem{example}[theorem]{Example}
\newtheorem{remark}[theorem]{Remark}

\newcommand{\mcS}{\mathcal{S}}

\newcommand{\mcN}{\mathcal{N}}

\renewcommand{\geq}{\geqslant}
\renewcommand{\leq}{\leqslant}

\begin{document}

\title{Wasserstein metrics and quantitative equidistribution of
  exponential sums over finite fields}

\author{E. Kowalski}
\address{ETH Z\"urich -- D-MATH\\
  R\"amistrasse 101\\
  8092 Z\"urich\\
  Switzerland} 
\email{kowalski@math.ethz.ch}

\author{T. Untrau}
\address{Univ Rennes, CNRS, IRMAR - UMR 6625, F-35000 Rennes, France} 
\email{theo.untrau@ens-rennes.fr}

\date{\today} 



\begin{abstract}
  The Wasserstein distance between probability measures on compact
  spaces provides a natural \enquote{invariant} quantitative measure of
  equidistribution, which is somewhat similar to the classical
  discrepancy appearing in Erd\H os--Turán type inequalities in the
  case of tori, but is a more intrinsic quantity.  We recall the basic
  properties of Wasserstein distances and present applications to
  quantitative forms of equidistribution of exponential sums in two
  examples, one related to our previous work on the equidistribution of
  ultra-short exponential sums, and the second a quantitative form of
  the equidistribution theorems of Deligne and Katz.
\end{abstract}

\maketitle
\setcounter{tocdepth}{1}

\tableofcontents

\section{Quantitative equidistribution and Wasserstein distances}

\subsection{Introduction}

Starting from the work of Weyl~\cite{weyl} about a century ago,
equidistribution has been a major theme of modern number theory.
Besides the qualitative aspect, there is considerable interest in having
quantitative versions of equidistribution theorems. In the most
classical case which concerns the equidistribution modulo~$1$ of a
sequence $(x_n)_{n\geq 1}$ of real numbers in $[0,1]$, equidistribution
is the property that
\begin{equation}\label{eq-equi}
  \lim_{N\to+\infty}\frac{1}{N} |\{n\leq N\,\mid\, a<x_n<b\}|=b-a
\end{equation}
for any real numbers~$a$ and~$b$, with $0\leq a<b\leq 1$. The Weyl
Criterion states that this is equivalent to the fact that the Weyl sums
\[
  W_h(N)=\frac{1}{N}\sum_{n\leq N}e(hx_n),\quad\quad h\in\Zz,\quad\quad
  e(z)=e^{2i\pi z}
\]
converge to~$0$ as $N\to +\infty$ for any non-zero integer~$h$.

The simplest quantitative forms of this equidistribution property are
simply quantitative estimates for the decay to~$0$ of the Weyl sums,
which should ideally be uniform in terms of $h$ and~$N$. Such
estimates also provide a quantitative version of~(\ref{eq-equi}), by
means of the classical \emph{Erd\H os--Turán inequality}: if we denote
\[
  E_N(a,b)=\frac{1}{N} |\{n\leq N\,\mid\, a<x_n<b\}|,
\]
then we have
\[
  \sup_{0\leq a<b\leq 1}|E_N(a,b)-(b-a)|\ll
  \frac{1}{T}+\sum_{0<|h|\leq T} \frac{|W_h(N)|}{|h|}
\]
for any parameter~$T\geq 1$, where the implied constant is absolute
(see~\cite[Th.\,III]{erdos-turan} for the original paper).

The left-hand side of this inequality is called the \emph{discrepancy}
of the sequence $(x_n)_{1\leq n\leq N}$, and is a natural measure of the
distance between the probability measure
\[
  \frac{1}{N}\sum_{n\leq N}\delta_{x_n}
\]
and the Lebesgue measure.

On the other hand, many problems give rise to multi-dimensional
equidistribution results, where the sequence $(x_n)$ takes values (for
instance) in $[0,1]^k$ for some integer~$k\geq 1$. Equidistribution
(with respect to the Lebesgue measure) means that
\[
  \lim_{N\to+\infty}\frac{1}{N} |\{n\leq N\,\mid\, a_i<x_{n,i}<b_i\text{
    for all } i\}|= \prod_i(b_i-a_i)
\]
whenever $0\leq a_i<b_i\leq 1$ for all~$i$, and is again equivalent to
the convergence towards $0$ of the Weyl sums
\[
  W_{\uple{h}}(N)=\frac{1}{N}\sum_{n\leq N}e(\uple{h}\cdot x_n)
\]
for $\uple{h}\in\Zz^k \setminus \{0\}$, where $x\cdot y$ denotes the
usual scalar product on~$\Rr^k$. The analogue of the Erd\H os--Turán
inequality is due to Koksma and states that the \enquote{box
  discrepancy}
\[
  \Delta_k=\sup_{\substack{0\leq a_i\leq b_i\leq 1\\1\leq i\leq k}}
  \Bigl| \frac{1}{N}|\{n\leq N\,\mid\, x_{n}\in
  \prod_{i=1}^k[a_i,b_i]\}| -\prod_{i=1}^k(b_i-a_i) \Bigr|
\]
satisfies
\begin{equation}\label{eq-erdos-turan-general}
  \Delta_k
  \ll \frac{1}{T}+\sum_{1\leq \|\uple{h}\|_{\infty}\leq T}
  \frac{1}{M(\uple{h})}|W_{\uple{h}}(N)|
\end{equation}
for any parameter~$T\geq 1$, where $M(\uple{h})=\prod\max(1,|h_i|)$ and
the implied constant depends only on~$k$ (see e.g.
\cite[Th. 1.21]{dt}).

However, this inequality is often unsatisfactory when~$d\geq 2$, due to
its lack of \enquote{invariance}. By this, we mean that if we apply to
the sequence $(x_n)$ a continuous map $f\colon [0,1]^k\to [0,1]^{k'}$,
or $f\colon [0,1]^k\to\Cc$, then the equidistribution of~$(x_n)$ with
respect to some measure~$\mu$ implies that~$(f(x_n))$ is equidistributed
with respect to the image measure $f_*(\mu)$, and one naturally wants to
have a quantitative version of this secondary convergence statement. But
this cannot be deduced from~(\ref{eq-erdos-turan-general}) without some
analysis of the way the map~$f$ transforms boxes with sides parallel to
the coordinate axes. Note that this is a problem even if~$f$ is (say) a
linear transformation in~$\SL_k(\Zz)$, and~$\mu$ is the Lebesgue
measure.  In particular, if the space $[0,1]^k$ arises abstractly
without specific choices of coordinates (which may well happen), then
the Erd\H os--Turán--Koksma inequality imposes the choice of a
coordinate system, which might be artificial and awkward for other
purposes. Moreover, if we consider the functional definition of
equidistribution, namely the fact that
\[
  \lim_{N\to +\infty} \frac{1}{N}\sum_{n\leq
    N}\varphi(x_n)=\int_{[0,1]^k}\varphi(x)dx
\]
for any continuous function $\varphi\colon [0,1]^k\to \Cc$, it is
difficult to pass from~(\ref{eq-erdos-turan-general}) to quantitative
bounds for the difference
\[
  \Bigl|\frac{1}{N}\sum_{n\leq
    N}\varphi(x_n)-\int_{[0,1]^k}\varphi(x)dx\Bigr|,
\]
which may be equally relevant to applications.

\subsection{Wasserstein distances}

From the probabilistic point of view, the goal of measuring the distance
between probability measures is very classical, and has appeared in many
forms since the beginning of modern probability theory. In recent years,
there has been increasing interest in using \emph{Wasserstein
  metrics}\footnote{\ Also called Monge--Kantorovich distances, or
  Kantorovich--Rubinstein metrics, see~\cite[p.\,206]{villani}.} for
this purpose. Indeed, these distances have considerable impact in
probability theory, statistics, the theory of partial differential
equations and numerical analysis.  However, they only appeared
relatively recently in works related to equidistribution questions,
including in works of Bobkov and Ledoux~\cite{bl_gaussian},
Steinerberger~\cite{steinerberger}, Brown and
Steinerberger~\cite{brown-steinerberger}, Graham~\cite{graham} and
Borda~\cite{borda_bernoulli,borda_berry_esseen}. Although not yet
well-established in the analytic number theory community, we believe
that the Wasserstein metrics provide a particularly well-suited approach
to quantitative equidistribution.

To explain this, we recall the definition of the Wasserstein distances
on a metric space $(M,d)$. Let~$p\geq 1$ be a real number.\footnote{\
  Later, $p$ will be used to also denote prime numbers, but there
  should be no ambiguity since only $\wass_1$ will appear then.}  For
Borel probability measures $\mu$ and~$\nu$ on $M$, let $\Pi(\mu,\nu)$
be the set of probability measures on $M\times M$ with marginals $\mu$
and~$\nu$ (which is not empty since $\mu\otimes\nu$ belongs to it).
The $p$-Wasserstein distance between~$\mu$ and~$\nu$ is then defined
by
\[
  \wass_p(\mu,\nu)=\inf_{\pi\in \Pi(\mu,\nu)}\Bigl(
  \int_{M \times M}d(x,y)^pd\pi(x,y)\Bigr)^{1/p}.
\]

\begin{remark}
  (1) This quantity depends on the choice of the metric~$d$. When
  needed, we will write $\wass_p^{(d)}(\mu,\nu)$ to indicate which
  metric is used.

  (2) In probabilistic terms, we have
  \[
    \wass_p(\mu,\nu)^p=\inf_{\substack{X\sim
        \mu\\Y\sim\nu}}\expect(d(X,Y)^p),
  \]
  where the infimum is taken over families of random variables $(X,Y)$
  with values in~$M$ such that the law of~$X$ is~$\mu$ and that of~$Y$
  is~$\nu$.
\end{remark}

The Wasserstein distances are particularly important in the theory of
optimal transport: indeed, they measure the cost of ``moving'' $\mu$
to $\nu$ (see, e.g. the book~\cite{villani} of Villani for an
introduction to optimal transport). Crucially from our point of view,
the definition of $\wass_p$ is intrinsic and does not suffer from the
same invariance issues as the box discrepancy. The key points are the
following:
\begin{enumerate}
\item the Wasserstein distances are metrics on the set of probability
  measures on~$M$, and if~$(M,d)$ is compact, then the topology that
  they define on this set is the topology of convergence in law;
\item the Wasserstein metrics satisfy simple inequalities under
  various operations, such as pushforward by a Lipschitz map;
\item the Wasserstein metric $\wass_1$ admits a very clean functional
  interpretation, known as Kantorovich--Rubinstein duality;
\item the Wasserstein metric $\wass_1$ also satisfies inequalities in
  terms of Weyl sums (in contexts much more general than that of
  $M=(\Rr/\Zz)^d$ above) which are comparable
  to~(\ref{eq-erdos-turan-general}).
\end{enumerate}

We summarize these basic properties.

\begin{theorem}[Wasserstein distance properties]\label{th-wass}
  Let $(M,d)$ be a complete separable metric space.

  \begin{enumth}
  \item If~$M$ is compact then, for any fixed real number $p\geq 1$, the
    Wasserstein metric is a metric on the space of probability measures
    on~$M$, and the topology it defines is the topology of convergence
    in law, i.e., $\wass_p(\mu_n,\mu)\to 0$ if and only if, for all
    (bounded) continuous functions $f\colon M\to\Cc$, we have
    \[
      \lim_{n\to+\infty}\int_M fd\mu_n=\int_Mfd\mu.
    \]
    
  \item For probability measures $\mu$ and~$\nu$ on $M$, and for any
    real numbers $q\geq p\geq 1$, we have
    \[
      \wass_p(\mu,\nu)\leq \wass_{q}(\mu,\nu)
    \]
    and if~$M$ has finite diameter, then
    \begin{equation}\label{eq-pq-bound}
      \wass_q(\mu,\nu)\leq \diam(M)^{1-p/q}\wass_p(\mu,\nu)^{p/q}.
    \end{equation}
        
  \item Let~$(N,\delta)$ be a metric space. Let~$c\geq 0$ be a real
    number and let~$f\colon M\to N$ be a $c$-Lipschitz map. For any
    probability measures~$\mu$ and~$\nu$ on~$M$, we have
    \[
      \wass_p(f_*\mu,f_*\nu)\leq c\wass_p(\mu,\nu).
    \]

  \item If $N\subset M$ is a closed subset with inclusion
    $i\colon N\to M$, then for probability measures~$\mu$ and~$\nu$
    on~$N$, we have
    \[
      \wass_p(i_*\mu,i_*\nu)=\wass_p(\mu,\nu),
    \]
    where the right-hand side is a Wasserstein distance on~$N$ with
    the induced metric from~$M$.

  \item Let~$(N,\delta)$ be a complete separable metric
    space. Let~$\Delta$ be a metric on~$M\times N$ such that there
    exists~$c>0$ such that
    \[
      \Delta((x,y),(x',y'))\leq c(d(x,x')+\delta(y,y'))
    \]
    for $(x,y)$ and~$(x',y')$ in $M\times N$. For any probability
    measures $\mu$ and~$\nu$ on~$M$, $\mu'$ and $\nu'$ on~$N$, we have
    \[
      \wass_p^{(\Delta)}(\mu\otimes\mu',\nu\otimes\nu')
      \leq c2^{1/q}(\wass_p(\mu,\nu)+\wass_p(\mu',\nu'))
    \]
    for $p\geq 1$, where~$q\geq 1$ is the dual exponent with
    $1/p+1/q=1$, and we use the convention $2^{1/\infty}=1$.
    
  \item For probability measures $\mu$ and~$\nu$ on $M$, we have
    \[
      \wass_1(\mu,\nu)=\sup_{u}\Bigl|\int_{M}ud\mu- \int_{M}ud\nu\Bigr|
    \]
    where the supremum is over functions~$u\colon M\to \Rr$ which are
    $1$-Lipschitz.

  \item Suppose that~$M=(\Rr/\Zz)^k$ with its standard metric for some
    integer~$k\geq 1$. For any probability measures $\mu$ and $\nu$
    on~$M$ and for all $T >0$, we have
    \begin{equation}\label{eq-bl}
      \wass_1(\mu,\nu)\leq \frac{4\sqrt{3}\sqrt{k}}{T} +\Bigl(\sum_{1\leq
        |\uple{h}|_{\infty}\leq T} \frac{1}{\|\uple{h}\|^2}
      |\widehat{\mu}(\uple{h})-\widehat{\nu}(\uple{h})|^2
      \Bigr)^{1/2},
    \end{equation}
    where for a probability measure $\mu$ on~$M$ and
    $\uple{h}\in\Zz^k$, we denote by $\widehat{\mu}(\uple{h})$ the
    Fourier coefficient of~$\mu$, i.e.
    \[
      \widehat{\mu}(\uple{h})=\int_{M}e(\uple{h}\cdot x)d\mu(x).
    \]
  \end{enumth}
\end{theorem}

\begin{proof}
  (1) This is proved, e.g., in Villani's
  book~\cite[Th.\,7.3,\,Th.\,7.12]{villani}.
  
  (2) The first inequality follows easily from H\"older's inequality and
  the definition of the Wasserstein metric
  (see~\cite[7.1.2]{villani}). For~(\ref{eq-pq-bound}),
  let~$\pi\in\Pi(\mu,\nu)$; then we obtain
  \[
    \int_{M\times M}d(x,y)^qd\pi\leq \Bigl(\sup_{(x,y)\in M\times
      M}d(x,y)\Bigr)^{q-p} \int_{M\times M}d(x,y)^pd\pi
    \leq\diam(M)^{q-p}\wass_p(\mu,\nu)^p,
  \]
  hence the result.

  (3) This is also a formal consequence of the definition and the fact
  that, given $\pi\in \Pi(\mu,\nu)$, we have
  $(f \times f)_*\pi\in \Pi(f_*\mu,f_*\nu)$. Then, for
  $\pi\in \Pi(\mu,\nu)$, we get
  \[
    \int_{N\times N}\delta(x,y)^pd(f\times f)_*\pi=\int_{M \times M}
    \delta(f(x),f(y))^pd\pi \leq c^p\int_{M \times M}d(x,y)^pd\pi,
  \]
  and taking the infimum over~$\pi\in \Pi(\mu,\nu)$ gives the
  inequality.

  (4) This follows from (3) and the fact that any measure in
  $\Pi(i_*\mu,i_*\nu)$ has support in~$N\times N$, hence is of the form
  $(i\times i)_*\pi$ for some $\pi\in \Pi(\mu,\nu)$.

  (5) This is also elementary: if $\pi\in\Pi(\mu,\nu)$ and
  $\pi'\in \Pi(\mu',\nu')$, then
  $\pi\otimes \pi'\in \Pi(\mu\otimes\mu',\nu\otimes\nu')$, with
  \[
    \int_{(M\times N)^2}\Delta((x,y),(x',y'))^pd(\pi\otimes\pi') \leq
    c^p\int_{(M\times N)^2}(d(x,x')+\delta(y,y'))^pd(\pi\otimes\pi')
  \]
  by assumption. Using the H\"older inequality in the form
  $(a+b)^p\leq 2^{p/q}(a^p+b^p)$ for $a$, $b\geq 0$, we obtain
  \[
    \int_{(M\times N)^2}\Delta((x,y),(x',y'))^pd(\pi\otimes\pi')\leq
    c^p2^{p/q}\Bigl( \int_{M^2} d(x,x')^pd\pi +
    \int_{N^2}\delta(y,y')^pd\pi'\Bigr),
  \]
  hence the result.

  (6) This is much deeper and is a special case of the
  Kantorovich--Rubinstein duality (see,
  e.g.,~\cite[Th.\,1.14]{villani}).

  (7) This is (essentially) an application of a statement due to
  Bobkov and Ledoux~\cite{bl_gaussian}. For completeness, we will give
  a proof in the Appendix starting on page~\pageref{sec: appendix proof}.
\end{proof}

Thus, as an alternative to the classical
inequality~(\ref{eq-erdos-turan-general}), we have the following
corollary.

\begin{corollary}\label{cor-wasserstein}
  Let~$k\geq 1$ be an integer.  Let $(x_n)_{n\geq 1}$ be a sequence in
  $(\Rr/\Zz)^k$. Denote by $\lambda_k$ the Lebesgue measure on
  $(\Rr/\Zz)^k$, and define the probability measures
  \[
    \mu_N=\frac{1}{N}\sum_{n=1}^N\delta_{x_n},
  \]
  for $N\geq 1$.  For any~$N\geq 1$ and any $T > 0$, we have
  \[
    \wass_1(\mu_N,\lambda_k)\leq \frac{4\sqrt{3}\sqrt{k}}{T}
    +\Bigl(\sum_{1\leq |\uple{h}|_{\infty}\leq T}
    \frac{1}{\|\uple{h}\|^2} |W_{\uple{h}}(N)|^2 \Bigr)^{1/2}.
  \]
\end{corollary}

\begin{remark}
  The Bobkov--Ledoux inequality was generalized by
  Borda~\cite{borda_berry_esseen} to arbitrary connected compact Lie
  groups. Since it requires extra notation and setup, we postpone the
  statement to Section~\ref{sec-deligne}, where it will be used to
  provide quantitative forms of equidistributions theorems for
  exponential sums over finite fields.
\end{remark}

We finally remark that for measures on~$\Rr$ (with the usual distance),
there is an ``explicit'' formula for the Wasserstein distances.

\begin{proposition}\label{pr-dim1}
  Suppose that~$M=\Rr$ with the euclidean distance. For any~$p\geq 1$
  and any probability measures~$\mu$ and~$\nu$ on~$\Rr$, we have
  \[
    \wass_1(\mu,\nu)=\int_{\Rr}
    |F_{\mu}(x)-F_{\nu}(x)|dx,\quad
    \wass_p(\mu,\nu)=\Bigl(\int_{0}^1
    |F^{-1}_{\mu}(x)-F^{-1}_{\nu}(x)|^pdx\Bigr)^{1/p},
  \]
  where
  \[
    F_{\mu}(x)=\mu(\mathopen]-\infty,x]),\quad\quad
    F_{\mu}^{-1}(t)=\inf\{x\in\Rr\,\mid\, F(x)>t\}
  \]
  are the cumulative distribution function of~$\mu$ and its
  generalized inverse, and similarly for~$\nu$.
\end{proposition}

This follows, e.g., from~\cite[Th.\,2.18,\,
Remark\,2.19\,(ii)]{villani}.

\subsection{Outline of the paper and further questions}

The remainder of this paper is structured as follows.  In
Section~\ref{sec-example}, we give a very simple illustration of the
optimal transport interpretation of Wasserstein distance in the simplest
possible case; this section may safely be skipped, but we believe that
it can provide some first insight for readers not yet experienced with
Wasserstein metrics. The main results are contained in
Sections~\ref{sec-ultra-short} and~\ref{sec-deligne}, which can be read
independently, and where we provide examples of quantitative arithmetic
equidistribution statements measured using Wasserstein distances. The
first refines our earlier work~\cite{ultrashort} on exponential sums
over roots of polynomials over finite fields (generalizing Gaussian
periods), while the second gives a quantitative version of
equidistribution theorems for trace functions over finite fields
(Deligne's Theorem~\cite{weil2}, as well as Katz's
Theorem~\cite{mellin}). We refer to the introductions of these separate
sections for the precise statements.

We also point out that a recent preprint of Cornelissen, Hokken and
Ringeling~\cite[Th.\,B]{chr} gives an application of Wasserstein
distance for a concrete arithmetic question (estimating asymptotically
the Mahler measure of certain Gaussian periods), related to the
content of Section~\ref{sec-ultra-short} (and using some of our
results).

We conclude this section with some questions which appear natural to us:
\begin{enumerate}
\item Is there (in general, or in specific cases like those of
  Sections~\ref{sec-ultra-short} and~\ref{sec-deligne}) some arithmetic
  interpretation of the measures on the relevant product space
  $M\times M$ which ``compute'' the Wasserstein distance? (Optimal
  measures are known to always exist in our setting, although they may
  not be unique; see, e.g.,~\cite[Th.\,1.3]{villani}.)
\item In a similar perspective, are there interesting arithmetic
  applications of the interpretation of the Wasserstein metric in
  terms of optimal transport?  For instance, the
  paper~\cite{borda_bernoulli} of Borda shows how assumptions of
  suitably ``small'' Wasserstein distances lead to ergodic-theoretic
  statements for random walks on compact groups, by exploiting the
  transport definition itself.
\item In arithmetic applications such as those discussed later, can one
  also give upper bounds for the $p$-Wasserstein distance for values
  $p\not=1$? The ``trivial'' estimate~(\ref{eq-pq-bound}),
  provides the bound
  \[
    \wass_p(\mu,\nu)\leq \diam(M)^{1-1/p}\wass_1(\mu,\nu)^{1/p},
  \]
  for $p\geq 1$, but one can hope that this can be improved (or it would
  be interesting in applications to know that this is not the case).
  
  The case $p=2$ seems particularly interesting in view of its
  properties in optimal transport (see for
  instance~\cite[Ch.\,2,\,8]{villani}).

  Although there is some form of duality for~$p>1$
  (see~\cite[Th.\,1.3]{villani}), this does not seem to translate in a
  simple functional description of the $p$-Wasserstein distance, so that
  estimates in terms of these metrics have a different flavor as those
  for the $1$-Wasserstein distance.
\item Wasserstein metrics make sense in considerable generality,
  including certain infinite-dimensional situations, such as separable
  Banach spaces. There are equidistribution theorems of arithmetic
  nature in such settings, one example being the convergence theorem
  for Kloosterman paths of Kowalski and Sawin~\cite{kloospaths}, where
  the underlying space is the Banach space $\mathcal{C}([0,1])$. It
  would be interesting to have a quantitative version of such a
  statement in terms of Wasserstein metrics.
\end{enumerate}

\subsection*{Acknowledgements}

E.K. was partially supported by the joint ANR-SNF project
``Equi\-distribution in Number Theory'', FNS grant number 10.003.145 and
ANR-24-CE93-0016.

We thank G. Tenenbaum for suggesting that the properties of the Weyl
sums in our equidistribution problem would lead to strong quantitative
versions, and C. Pilatte for suggesting to present the results of
Section~\ref{sec-ultra-short} in terms of finite sets of algebraic
integers. We are also thankful to B. Borda for pointing out the two
versions of the inequality of Bobkov and Ledoux and for very helpful
comments on an earlier version of this article. We thank P. Nelson for
discussions concerning representations of compact Lie groups.  The
second author also wishes to thank I.  Shparlinski for sending an
unpublished note that contained interesting ideas and led him to study
the question of Section~\ref{sec: growing d}. He would also like to
thank F. Jouve, G. Ricotta and R. Tichy for encouragements, as well as
F. Bolley, T.  Cavallazzi, B. Huguet and T. Modeste for taking the time
to answer his questions on probability theory.

\section{A simple example}\label{sec-example}

For illustration, it is worth looking concretely at the meaning of the
Wasserstein distance between a measure of the type
\[
  \mu_N=\frac{1}{N}\sum_{i=1}^N\delta_{x_i}
\]
on $\Rr/\Zz$ and the uniform Lebesgue measure~$\lambda$.\footnote{\ The
  computation also follows from Proposition~\ref{pr-dim1}, but it is
  still worth looking at it concretely.} We assume for simplicity that
the $x_i$'s are different, and
denote~$\mathcal{X}=\{x_i\,\mid\, 1 \leq i \leq N\}$. Any measure
$\pi\in \Pi(\mu_N,\lambda)$ must be supported in
$\mathcal{X}\times \Rr/\Zz\subset (\Rr/\Zz)^2$, since $\mathcal{X}$ is
the support of~$\mu_N$. Hence we may write
\[
  \pi=\frac{1}{N}\sum_{i=1}^N\delta_{x_i}\otimes \pi_i,
\]
for some probability measure~$\pi_i$ on $\Rr/\Zz$. Any such measure
has projection on the first coordinate equal to~$\mu_N$. On the
other hand, the projection on the second coordinate is the combination
\[
  \frac{1}{N}\sum_{i=1}^N\pi_i,
\]
and thus $\pi\in\Pi(\mu_N,\lambda)$ is equivalent with
\[
  \frac{1}{N}\sum_{i=1}^N\pi_i=\lambda.
\]

In particular, suppose that~$\pi_i=f_i(y)dy$ for some measurable
function $f_i\geq 0$ with integral equal to~$1$. Then the condition
above becomes
\[
  \frac{1}{N}\sum_{i=1}^Nf_i=1.
\]

The corresponding integral for the $p$-Wasserstein distance is
\[
  \int_{(\Rr/\Zz)^2}d(x,y)^pd\pi=
  \frac{1}{N}\sum_{i=1}^N \int_{\Rr/\Zz}d(x_i,y)^pf_i(y)dy.
\]


As a simple illustration, let~$N\geq 2$ and consider~$x_i=(i-1)/N$
for $1\leq i\leq N$. Let~$\varphi_i$ be the indicator function of
the (image modulo $\Zz$ of the) interval
\[
  \Bigl[\frac{i-1}{N}-\frac{1}{2N},
  \frac{i-1}{N}+\frac{1}{2N}\Bigr[,
\]
and let~$f_i=N\varphi_i$. Then $f_i$ has integral~$1$, and
\[
  \sum_i f_i=N
\]
since any $x\in\Rr/\Zz$ belongs to a single interval, where the
corresponding $f_i$ takes the value~$N$. For the measure~$\pi$
described above, using the usual circular distance on~$\Rr/\Zz$, we
get
\[
  \int_{(\Rr/\Zz)^2}d(x,y)^pd\pi= \frac{1}{N}\sum_{i=1}^N
  \int_{\Rr/\Zz}d(x_i,y)^pf_i(y)dy=
  \sum_{i=1}^N\int_{(i-1)/N-1/(2N)}^{(i-1)/N+1/(2N)}
  \Bigl|\frac{i-1}{N}-y\Bigr|^pdy.
\]

By a simple integration, this is
\[
  \frac{2N}{p+1}\Bigl(\frac{1}{2N}\Bigr)^{p+1}=
  \frac{1}{(p+1)2^p}\frac{1}{N^p},
\]
proving that
\[
  \wass_p\Bigl(\frac{1}{N}\sum_{i=1}^N\delta_{(i-1)/N},\lambda\Bigr)
  \leq \frac{1}{2(p+1)^{1/p}}\frac{1}{N}\leq\frac{1}{2N}
\]
for all~$p\geq 1$. For $p=1$, this result is of the same quality as
the outcome of Corollary~\ref{cor-wasserstein}, since in this case the
Weyl sum $W_h(N)$ is zero unless $N\mid h$, in which case it is equal
to~$1$, so that the estimate in loc. cit. is
\[
  \wass_1(\mu_N,\lambda)\leq \frac{4\sqrt{3}}{T}+
  \Bigl(\sum_{\substack{1\leq |h|\leq T\\ N\mid
      h}}\frac{1}{|h|^2}\Big)^{1/2} \ll \frac{1}{T}+\frac{1}{N},
\]
and taking $T=N$ gives the result. This is of course comparable to the
well-known fact that the discrepancy of~$N$ equally-spaced points is of
size $1/N$.

Note that as a consequence of Theorem~\ref{th-wass}, (5), we deduce
also that for the uniform measure~$\mu_{N}^{(k)}$ on the finite grid
\[
  \{0, 1/N,\ldots, (N-1)/N\}^k\subset (\Rr/\Zz)^k,
\]
we have
\[
  \wass_p(\mu_N^{(k)},\lambda_k)\ll \frac{1}{N}
\]
where the implied constant depends on~$k$ only (and the Wasserstein
distance is computed for the standard metric).  Here again, taking
$T=N$ in Theorem~\ref{th-wass}, (7) gives a bound of the same shape
for $p=1$.

\section{Ultra-short sums of trace functions}\label{sec-ultra-short}

\subsection{Statement of the result}

Our previous paper~\cite{ultrashort} considers an equidistribution
problem on a torus for which the quantitative features of the
Wasserstein metric are particularly useful.  We considered there the
distribution properties of sums of the form
\begin{equation}\label{eq-sums}
  S_g(q,a) = \sum_{\substack{x \in \Ff_q \\ g(x) \equiv 0 \mods{q}}}^{}
  e\Bigl(\frac{ax}{q}\Bigr),
\end{equation}
where $g \in \Zz[X]$ is a fixed monic polynomial and $q$ is a prime
number totally split in the splitting field of~$g$ (subject to suitable
conditions). As suggested by C. Pilatte, it is equally straightforward
to formulate the results in terms of arbitrary finite sets of algebraic
integers, and we will do so.

Thus let $Z\subset \Cc$ be a finite set of algebraic integers. We denote
by $K_Z=\Qq(Z)$ the number field generated by~$Z$ and by $\Oo_Z$ the
ring of integers of~$K_Z$. For any prime ideal $p\subset \Oo_Z$ which
has residual degree~$1$, and for $a\in\Oo_Z/p$, we denote
\[
  S_Z(p,a) = \sum_{x\in Z} e\Bigl(\frac{a\varpi_p(x)}{|p|}\Bigr),
\]
where $|p|$ is the norm of~$p$ and
$\varpi_p\colon \Oo_Z\to \Oo_Z/p\simeq \Ff_{|p|}$ is the reduction
map. We then form the probability measure
\[
  \nu_p = \frac{1}{|p|} \sum_{a \in \Oo_Z /p} \delta_{S_Z(p,a)}
\]
on~$\Cc$.  In \cite[Th.\,1.1]{ultrashort}, we assumed that~$Z$ is the
set of complex zeros of a monic separable polynomial, and proved that
these measures converge weakly to an explicit probability measure,
which is related to the additive
relations among the zeros of $g$. We will prove a similar statement for
this slightly more general case and refine it by deriving a rate of
convergence using the $1$-Wasserstein metric on~$\Cc$. Precisely, let
\[
  R_Z=\{\alpha\colon Z\to\Zz\,\mid\, \sum_{x\in Z}\alpha(x)x=0\},
\]
and let
\begin{equation}\label{eq-hz}
  H_Z=\{f\colon Z\to\Ss^1\,\mid\, \alpha\in R_Z\text{ implies }
  \prod_{x\in Z}f(x)^{\alpha(x)}=1\}.
\end{equation}

Denote by $\mu_Z$ the direct image by the map
\[
  \sigma \colon f \mapsto \sum_{x\in Z}f(x)
\]
of the probability Haar measure on the compact group~$H_Z$. We will
prove:

\begin{theorem}[Cor. \ref{cor: quantitative in complex plane}]
  \label{th: rate equi}
  With notation as above, for prime ideals~$p\subset \Oo_Z$ of residual
  degree~$1$, we have
  \[
    \wass_1(\nu_p,\mu_Z) \ll_Z |p|^{-\frac{1}{[K_Z: \Qq]}}.
  \]
\end{theorem}

The quantitative results in Wasserstein metrics also allow us to
consider some problems involving \emph{varying} sets (or zeros of
varying polynomials). As an example, in Section \ref{sec: growing d} we
exploit the explicit dependency on~$k$ in the inequality of
Theorem~\ref{th-wass}, (7) to study the distribution of exponential sums
over very small multiplicative subgroups of prime order, i.e. sums of
the form
\[
  \sum_{x \in H} e\left(\frac{ax}{q}\right)
\]
where $H$ is a subgroup of $\Ff_q^{\times}$ whose cardinality is a
\emph{prime} divisor of $q-1$ which is very small compared to
$q$. Precisely, we prove in Theorem~\ref{th: growing subgroups} that the
sums
\[
  \frac{1}{\sqrt{|H|}} \sum_{x \in H} e\left(\frac{ax}{q}\right)
\]
become equidistributed with respect to a standard complex normal
distribution as $q$ tends to infinity and $|H|$ tends to infinity
while satisfying
\[
  |H| = o\left( \frac{\log q}{\log\log q}
  \right)
\]
as $q\to+\infty$.

\subsection{Proof of Theorem~\ref{th: rate equi}}\label{sec: ultra-short}

We use the notation (e.g., $Z$, $K_Z$, $\Oo_Z$, etc) of the statement of
the theorem.  For any set $X$, we denote by $C(Z; X)$ the set of maps
from $Z$ to $X$; it is a group if~$X$ itself is a group. For all
$\alpha \in C(Z; \Zz)$, we set
$$
\|\alpha \|_1 = \sum_{x \in Z} |\alpha(x)|, \quad \| \alpha \|_{\infty}
= \max_{ x \in Z}|\alpha(x)|, \quad \text{and}\quad |\alpha| =
\Bigl(\sum_{x \in Z}\alpha(x)^2\Bigr)^{1/2}.
$$

We denote by~$\lambda_Z$ the probability Haar measure on the compact
group~$H_Z$ defined in~(\ref{eq-hz}). To prove Theorem~\ref{th: rate
  equi}, we will (as in our previous paper) first prove equidistribution
and estimate the Wasserstein distance for a sequence of measures $\mu_p$
on $C(Z;\Ss^1)$, and then deduce the statement of the theorem by
considering the push-forward by the map~$\sigma$. 

For the first step, we consider on the compact group~$C(Z;\Ss^1)$ the
metric~$\rho$ obtained by transport of structure from the usual
flat Riemannian metric on the torus~$(\Rr/\Zz)^{Z}$ using the
isomorphism
\[
  \iota\colon (\Rr/\Zz)^{Z}\to C(Z;\Ss^1)
\]
such that $\iota(\uple{\xi})$ is the map $Z\to\Ss^1$ which maps $x\in Z$
to $e(\uple{\xi}_x)$. This metric on the torus is also the quotient
metric of the euclidean distance using the projection
$\Rr^{Z}\to (\Rr/\Zz)^{Z}$.

Concretely, this metric can also be described as follows when
identifying further $\Rr/\Zz$ with $[0,1[$: if~$\ell$ denotes the
``circular'' metric on~$[0,1[$, identified with~$\Rr/\Zz$, such that
$\ell(x,y)=\min(|x-y|,1-|x-y|)$, then the metric~$\rho$ can be
identified with the metric on~$[0,1[^{Z}$ given by
\begin{equation}\label{eq-rhod}
  (x,y)\mapsto \Bigl(\sum_{j=1}^{|Z|}\ell(x_j,y_j)^2\Bigr)^{1/2}.
\end{equation}

\subsection{The Weyl sums}
\label{sec: section vanishing weyl sums}

We denote by $\mcS_Z$ the set of prime ideals $p\subset \Oo_Z$ of
residual degree one such that no two distinct elements in~$Z$ have
difference in~$p$.

For $p\in\mcS_Z$, we denote by $q$ the norm $|p|=|\Oo_g/p|$; this is a
prime number, and the restriction $\Zz\to \Oo_Z/p$ of the reduction
map~$\varpi_{p}\colon \Oo_Z \to \Oo_Z /p$ induces a field isomorphism
$\Ff_q \to \Oo_Z/p$, which we use to identify these two fields.

We view $\Oo_Z/p$ as a finite probability space with the uniform
probability measure and we consider the random variables $U_{p}$
on~$\Oo_Z/p$, taking values in $C(Z;\Ss^1)$, which are defined by
\[
  U_{p}(a)(x)=e\Bigl(\frac{a\varpi_p(x)}{|p|}\Bigr)
\]
for all $a \in \Oo_Z/p$ and all $x \in Z$ (the element $a\varpi_p(x)$ of
$\Oo_Z/p$ being identified with an element of $\Ff_{|p|}=\Ff_q$ as
explained before).

Let~$\mu_p$ be the law of~$U_p$; concretely, this is the measure
$$
\mu_p = \frac{1}{|p|}\sum_{a \in \Oo_Z / p}\delta_{U_p(a)}
$$
on $C(Z; \Ss^1)$. We also view the Haar measure $\lambda_Z$ on~$H_Z$ as
a measure on~$C(Z;\Ss^1)$, supported on~$H_Z$, by identifying it with
$j_*\lambda_Z$, where $j\colon H_Z\to C(Z;\Ss^1)$ is the inclusion.

The next two lemmas will compare the Fourier coefficients of the
measures $\mu_p$ and~$\lambda_Z$.

For~$\alpha\in C(Z;\Zz)$, we will denote by~$\eta_{\alpha}$ the
associated character in~$C(Z;\Ss^1)$. For a measure~$\nu$
on~$C(Z;\Ss^1)$, the corresponding Fourier coefficient is
\[
  \widehat{\nu}(\alpha)=\int \eta_{\alpha}(x)d\nu(x).
\]

\begin{lemma}\label{lem: vanishing triv}
  Let $\alpha \in C(Z; \Zz)$. If $\eta_\alpha$ is trivial on $H_Z$, then
  $\hat{\lambda}_Z(\alpha) = \hat{\mu}_p(\alpha)= 1$.
\end{lemma}

\begin{proof}
  The result about $\lambda_Z$ simply follows from the fact that it has
  total mass~$1$. The case of $\mu_p$ follows because it is
  straightforward from the definition that the random variables $U_p$
  take values in $H_Z$.
\end{proof}

\begin{lemma}\label{lem: vanishing nontriv}
  There exists a positive real number $C_Z$, depending only on the
  set~$Z$, with the following property\textup{:} for all
  $\alpha \in C(Z; \Zz)$ such that $\eta_\alpha$ is non-trivial on
  $H_Z$ and for all~$p\in \mcS_Z$ such that
  \[
    |p| > C_Z \|\alpha \|_{1}^{[K_Z: \Qq]},
  \]
  we have $\hat{\lambda}_Z(\alpha) = \widehat{\mu}_p(\alpha)=0$.
  More precisely, the value
  \begin{equation}\label{eq-cg}
    C_Z= \prod_{\sigma \colon K_Z\to\Cc}\max_{ x \in Z}
    |\sigma(x) |
  \end{equation}
  has this property, where the product ranges over the embeddings
  of~$K_Z$ in~$\Cc$.
\end{lemma}

\begin{proof}
  Let $\alpha \in C(Z; \Zz)$ be such that~$\eta_{\alpha}$ is non-trivial
  on~$H_Z$. It is then a classical property of the Haar measure that
  $\widehat{\lambda}_g(\alpha)=0$.

  Let~$p\in \mcS_Z$. Arguing as in~\cite[Prop.\,2.2]{ultrashort}, we
  have
  \[
    \widehat{\mu}_p(\alpha)=
    \frac{1}{|p|}\sum_{a\in\Oo_Z/p}e\Bigl(
    \frac{a}{|p|}\varpi_p\Bigl(\sum_{x\in Z}\alpha(x)x\Bigr)\Bigr)
    \Bigr),
  \]
  and therefore
  $$
  \hat{\mu}_p(\alpha) = \begin{cases}
    1 \quad \text{ if } \sum_{x\in Z}\alpha(x)x \in p, \\
    0 \quad \text{ otherwise, }
  \end{cases}
  $$
  by orthogonality of characters. We now analyze this condition
  further.

  Let
  \[
    \gamma(\alpha)=\sum_{x\in Z}\alpha(x)x.
  \]

  Note that since~$\eta_{\alpha}$ is non-trivial on~$H_Z$, we
  have~$\alpha\notin R_Z$ (by definition), and therefore
  $\gamma(\alpha)\not=0$.

  If~$\widehat{\mu}_p(\alpha)\not=0$, then $\gamma(\alpha) \in p$, so
  $\gamma(\alpha) \Oo_Z\subset p$, and in particular
  $|p|\mid N(\gamma(\alpha)\Oo_Z)$.  In particular, since
  $\gamma(\alpha)$ is non-zero, we obtain
  \begin{equation} \label{eq: minoration norme gammaalpha}
    |N_{K_Z/\Qq}(\gamma(\alpha))| \geqslant |p|.
  \end{equation}
  
  On the other hand, we have
  \[
    N_{K_Z/\Qq}(\gamma(\alpha))= \prod_{\sigma \colon K_Z\to\Cc}
    \sigma(\gamma(\alpha)) = \prod_{\sigma \colon K_Z\to \Cc} \Bigl(
  \sum_{x\in Z}^{} \alpha(x)\sigma(x) \Bigr),
  \]
  and with the value~$C_Z$ given by~(\ref{eq-cg}), this implies that
  \begin{equation}\label{eq: majoration norme gammaalpha}
    |N_{K_Z/\Qq}(\gamma(\alpha))| \leqslant C_Z \Bigl( \sum_{x \in
      Z}|\alpha(x) |\Bigr)^{[K_Z:\Qq]} = C_Z \| \alpha \|_1^{[K_Z: \Qq]}.
  \end{equation}
  
  Combining \eqref{eq: minoration norme gammaalpha} and \eqref{eq:
    majoration norme gammaalpha} we deduce that
  \[
    |p| \leq C_Z  \| \alpha \|_1^{[K_Z: \Qq]}
  \]
  if~$\widehat{\mu}_p(\alpha)\not=0$, which is the desired conclusion.
\end{proof}

\subsection{Quantitative equidistribution}\label{sec: quantitative rate}

We continue with the notation of the previous section.

\begin{proposition}\label{pr-wass-mug}
  There exists an explicit constant $C^{'}_Z >0$, depending only on $Z$,
  such that for all $p \in \mcS_Z$, the inequality
  \[
    \wass_1(\mu_p, \lambda_Z) \leq C'_Z |p|^{-\frac{1}{[K_Z: \Qq]}}
  \]
  holds.
\end{proposition}

\begin{proof}
  Let~$T>0$ be an auxiliary parameter to be fixed below. By
  Theorem~\ref{th-wass}, (7), applied to the space $C(Z;\Ss^1)$, the
  inequality
  \begin{equation}\label{eq: applying bl}
    \wass_1(\mu_p,\lambda_Z)
    \leqslant \frac{4 \sqrt{3}\sqrt{|Z|}}{T} +
    \Bigl(\sum_{\substack{\alpha \in C(Z; \Zz) \\ 0 <
        \|\alpha\|_{\infty} \leqslant T}}
    \frac{1}{|\alpha|^2}|\hat{\mu_p}(\alpha) -
    \hat{\lambda_Z}(\alpha)|^2\Bigr)^{1/2}
  \end{equation}
  holds. We take
  \[
    T = \frac{1}{|Z|+1}
    \left(\frac{|p|}{C_Z}\right)^{\frac{1}{[K_Z:\Qq]}}.
  \]

  Lemma~\ref{lem: vanishing triv} and Lemma~\ref{lem: vanishing nontriv}
  together imply that the sum on the right-hand side of~\eqref{eq:
    applying bl} is then zero (using the inequality
  $\| \alpha\|_1 \leqslant |Z| \|\alpha \|_{\infty}$), and therefore we
  obtain
  \[
    \wass_1(\mu_p, \lambda_Z) \leq \frac{4 \sqrt{3}\sqrt{|Z|}}{T} = 4
    \sqrt{3}\sqrt{|Z|} (|Z|+1) C_Z^{\frac{1}{[K_Z:\Qq]}}
    |p|^{-\frac{1}{[K_Z:\Qq]}},
  \]
  which immediately implies the result.
\end{proof}

\begin{rem}
  This result matches the rate of convergence obtained
  in~\cite[Th. 5.30]{mathese}, where the $1$-Wasserstein metric was
  replaced by a notion of $\varphi$-discrepancy, which had the
  disadvantage of being non-intrinsic.
\end{rem}

We can easily deduce Theorem~\ref{th: rate equi} using
Theorem~\ref{th-wass}, (3). First we compute a Lipschitz constant for
the summation map~$\sigma$.

\begin{lemma}\label{lem: pushforward}
  The map $\sigma\colon C(Z;\Ss^1)\to\Cc$ is $\sqrt{|Z|}$-Lipschitz.
\end{lemma}

\begin{proof}
  Let $f$ and~$g$ be elements of $C(Z; \Ss^1)$. We have
  \begin{align*}
    |\sigma(f) -\sigma (g)| = \Bigl|\sum_{x \in Z} (f(x)-
    g(x))\Bigr| \leqslant \sum_{x \in Z}\Bigl|f(x)- g(x)\Bigr|.
  \end{align*}

  The euclidean distance on~$\Ss^1\subset\Cc$ is bounded above by the
  arc length (Riemannian) distance $\ell$, hence
  $|f(x) - g(x)| \leq \ell(f(x), g(x))$. Applying the Cauchy--Schwarz
  inequality, we obtain
  $$
  \sum_{x \in Z} | f(x)- g(x)| \leq \sum_{x \in Z}\ell(
  f(x), g(x)) \leqslant \sqrt{|Z|} \Bigl( \sum_{x \in Z}
  \ell(f(x),g(x))^2\Bigr)^{1/2} = \sqrt{|Z|} \times \rho(f,g),
  $$
  as desired.
\end{proof}

\begin{corollary}\label{cor: quantitative in complex plane}
  With notation as above, for all prime ideals $p\in\mcS_Z$, the
  measures
  \[
    \nu_p = \frac{1}{|p|} \sum_{a \in \Oo_Z /p} \delta_{S_Z(p,a)}
  \]
  and $\mu_Z = \sigma_{*} \lambda_Z$ satisfy
  \[
    \wass_1(\nu_p, \mu_Z) \ll_Z |p|^{-\frac{1}{[K_Z: \Qq]}}.
  \]
\end{corollary}

\begin{proof}
  By Proposition~\ref{pr-wass-mug}, we have
  $\wass_1(\mu_p, \lambda_Z) \ll_Z|p|^{-\frac{1}{[K_Z: \Qq]}}$. It then
  follows from Theorem \ref{th-wass}, (3) and Lemma \ref{lem:
    pushforward} that
  $$
  \wass_1(\sigma_{*}\mu_p, \sigma_{*}\lambda_Z) \ll_Z
  |p|^{-\frac{1}{[K_Z: \Qq]}}
  $$
  (where the implied constant on the right-hand side incorporates the
  factor $\sqrt{|Z|}$).
\end{proof}

\begin{remark}
  (1) This corollary refines the qualitative convergence statement
  in~\cite[Cor. 2.4]{ultrashort}.

  (2) Let~$q$ be a prime number which is the norm of a prime ideal~$p$
  in $\mcS_Z$. Then $\Ff_q$ can be identified with $\Oo_Z/p$, and the
  reduction map $Z\to \Ff_q$ is injective. We can then define
  \[
    S_Z(q,a) = \sum_{x\in Z\mods{q}} e\Bigl(\frac{ax}{q}\Bigr),
  \]
  with an obvious meaning for $Z\mods{q}$.  The conclusion of the
  corollary then applies to the measures of the form
  \[
    \frac{1}{q} \sum_{a \in \Ff_{q}} \delta_{S_Z(q,a)},
  \]
  since this measure coincides with $\nu_p$.

  If the set $Z$ is the set of (all) zeros of a fixed monic polynomial
  $g\in\Zz[X]$, note that $S_Z(q,a)$ can be defined \emph{a priori}
  without mentioning $K_Z$ or $\Oo_Z$, since $Z\mods{q}$ is then the set
  of roots of~$g$ modulo~$q$.
\end{remark}

\subsection{Sums of additive characters over growing multiplicative
  subgroups} \label{sec: growing d}

So far, we have only been dealing with weak convergence of measures in
compact groups since we were essentially working in $(\Ss^1)^d$ for a
\emph{fixed $d$}. However, the Wasserstein metric metrizes weak
convergence in a much more general context, and in this section we
give an application in a non-compact setting.

We consider the sums
$$
S_d(q,a) = \sum_{x \in \mu_d(\Ff_q)}^{} e\left(\frac{ax}{q}\right)
$$
which are a special case of the sums $S_g(q,a)$ defined in \eqref{eq-sums} with $g = X^d -1$. For
a \emph{fixed} integer $d$ and $q$ going to infinity among the primes
congruent to $1$ modulo $d$, their asymptotic behaviour attracted
interest partly because of the beautiful visual aspect of the plots:
see e.g. \cite{gaussian_periods, burkhardt,menagerie,untrau, platt}
for examples and generalizations. In particular, if $d$ is prime
\cite[Th. 6.3]{gaussian_periods} states that they become
equidistributed (as $q$ goes to infinity and $a$ varies in $\Ff_q$)
with respect to the pushforward measure of the Haar probability
measure on $(\Ss^1)^{d-1}$ via the Laurent polynomial
$$
\begin{array}{ccccc}
  g_d& : &(\Ss^1)^{d-1} & \to & \Cc \\ &&
                                          (z_1, \dots, z_{d-1}) & \mapsto & z_1 + \cdots + z_{d-1} + \frac{1}{z_1 \dots z_{d-1}}
\end{array}
$$
This result explains why the sums appeared to fill out the region of
the complex plane delimited by a $d$-cusps hypocycloid.

We are now interested in studying the case where $d$ is allowed to
vary with $q$. In more probabilistic terms, the previous result says
that when $q$ is large the subsets $$\{S_d(q,a), \ a \in \Ff_q \}$$ of
the complex plane look like $q$ independent values taken by a sum of
the form $$Z_1 + \dots + Z_{d-1} + \frac{1}{Z_1 \dots Z_{d-1}}$$ of
independent and identically distributed Steinhaus random variables
$Z_i$ (i.e. uniform on $\Ss^1$). Thanks to the multidimensional
Central Limit Theorem, the random variables
$$
\frac{1}{\sqrt{d}} \left(Z_1 + \dots + Z_{d-1} +
  \frac{1}{Z_1 \dots Z_{d-1}}\right)
$$
converge in law to a two dimensional gaussian
$\mcN(0, \frac12 \mathrm{Id})$ (the coefficient $\frac12$ just comes
from the value of the variance of the real and imaginary parts of a
uniform random variable on the circle). Therefore, if we denote by
$\mu_{q,d}$ the measure
$$
\frac{1}{q} \sum_{a \in \Ff_q}^{} \delta_{\frac{1}{\sqrt{d}} S_d(q,a)}
$$
then
$\lim_{d \to \infty} \left(\lim_{q \to \infty} \mu_{q,d}\right) =
\mcN(0, \frac12 \mathrm{Id}).$ However
$\lim_{q \to \infty} \left(\lim_{d \to \infty} \mu_{q,d}\right) =
\delta_0,$ so we are interested in intermediate regimes, in which both
$q$ and $d$ tend to infinity, and the limit of the sequence of
measures $\mu_{q,d}$ can be determined. Equivalently, this means that
we are interested in the distribution of the sums
\begin{equation} \label{eq: sums S_d sur racine d} \frac{1}{\sqrt{d}}
  \sum_{x \in \mu_d(\Ff_q)}^{} e\left(\frac{ax}{q}\right)
\end{equation}
as $a$ varies in $\Ff_q$ and both $q$ and $d$ tend to infinity (with
$q$ and $d$ prime and $q \equiv 1 \mods{d}$). These sums of additive
characters over multiplicative subgroups whose cardinality grows with
$q$ have been studied before, mostly with the aim of proving
non-trivial upper bounds. In particular, when $d$ grows at least like
a small power of $q$, the groundbreaking work of Bourgain, Glibichuk
and Konyagin \cite{BGK} shows a power saving bound for $S_d(q,a)$. On
the other hand, \cite[Th. 1.8]{konyagin_lecture_notes} shows that if
$d \ll \log(q)$, it is impossible to obtain a non-trivial bound.

Combining the results of Section \ref{sec: quantitative rate} and the
Central Limit Theorem, we will prove the following result in a similar setting where $d$
is very small with respect to $q$.

\begin{theorem}\label{th: growing subgroups}
  For every odd prime $q$, we let $d= d(q)$ be a prime divisor of
  $q-1$. If $d(q) \underset{q \to +\infty}{\longrightarrow} + \infty$
  and
  $$
  d(q) \underset{q \to \infty}{=} o\left( \frac{\log q}{ \log \log
      q}\right)
  $$
  then as $q$ tends to infinity and $a$ varies in $\Ff_q$, the sums
  \eqref{eq: sums S_d sur racine d} become equidistributed in the
  complex plane with respect to a normal distribution
  $\mcN(0, \frac12 \mathrm{Id})$.
\end{theorem}

The first step of the proof consists in using the following
quantitative form of the convergence already obtained in
\cite[Th. 6.3]{gaussian_periods}.

\begin{lemma} \label{lem: quantitative duke garcia lutz} Let $d$ and
  $q$ be two prime numbers such that $q \equiv 1 \mods{d}$. Denote by
  $\gamma_d$ the pushforward measure of the Haar probability measure
  on $(\Ss^1)^{d-1}$ via the Laurent polynomial
  $\frac{1}{\sqrt{d}}g_d$ (in other words it is the law of the random
  variable
  $\frac{1}{\sqrt{d}}\left(Z_1 + \cdots + Z_{d-1} + \frac{1}{Z_1 \dots
      Z_{d-1}}\right)$ where the $Z_i$ are independent and identically
  distributed uniform random variables on $\Ss^1$).  Then
  $ \wass_1(\mu_{q,d}, \gamma_d) \leqslant 2 \sqrt{12} \sqrt{d} (d+1)
  q^{-\frac{1}{d-1}} .  $
\end{lemma}

\begin{proof}
  We first use Proposition~\ref{pr-wass-mug} in the particular case
  where $Z = \mmu_d$ (the set of $d$--th roots of
  unity in $\Cc$) and $K_Z$ is the cyclotomic field $\Qq(\mmu_d)$. In
  particular, $[K_Z: \Qq] = d-1$ and we can take $C_Z = 1$ in
  Lemma~\ref{lem: vanishing nontriv} because the roots of unity have
  modulus $1$. We recall that for all $p \in \mcS_{\mmu_d}$, the
  measure $\mu_p$ denotes the law of the random variable $U_p$ (as
  introduced in Section~\ref{sec: section vanishing weyl sums}), and
  $\lambda_{\mmu_d}$ is the Haar probability measure on $H_{\mmu_d}$. Then using the explicit constant $C'_Z$ obtained at the end
  of the proof of Proposition~\ref{pr-wass-mug}, we deduce that
  $$
  \wass_1(\mu_p, \lambda_{\mmu_d}) \leqslant 4 \sqrt{3} \sqrt{d} (d+1) |p|^{-\frac{1}{d-1}}.
  $$
  
  Then we pushforward via the map $\sigma_d=d^{-1/2}\sigma$, defined
  by
  \[
    \sigma_d(f)=\frac{1}{\sqrt{d}}\sum_{x\in \mmu_d}f(x).
  \]
  This map is
  $1$-Lipschitz thanks to Lemma \ref{lem: pushforward}. This gives
  \begin{equation} \label{eq: sigma_d*} \wass_1\left( \left(\sigma_d
      \right)_{*} \mu_p, \left(\sigma_d \right)_{*} \lambda_{\mmu_d}
    \right) \leqslant 4 \sqrt{3} \sqrt{d} (d+1) |p|^{-\frac{1}{d-1}}.
  \end{equation}

  Finally, in \cite[\S 3]{ultrashort} we showed that when
  $d$ is prime the $\Zz$-module
  $R_{\mmu_d}$ of additive relations among the roots of
  $X^d-1$ is generated by the constant map equal to $1$, so
  $H_{\mmu_d}$ can be identified with
  $(\Ss^1)^{d-1}$. Then it follows formally from the definitions that
  for $q \equiv 1 \mods{d}$ the measure
  $\mu_{q,d}$ coincides with the measure
  $(\sigma_d)_{*}\mu_p$ for all $p$ lying above $q$ (and such
  $p$ belong to
  $\mcS_{\mmu_d}$ thanks to \cite[Cor. 10.4]{neukirch}), and the measure
  $\gamma_d$ coincides with
  $(\sigma_d)_*\lambda_{\mmu_d}$. Therefore, \eqref{eq: sigma_d*}
  actually says that $ \wass_1(\mu_{q,d}, \gamma_d ) \leqslant 4
  \sqrt{3} \sqrt{d} (d+1) q^{-\frac{1}{d-1}}$.
\end{proof}

We will also need a form of the Central Limit Theorem in Wasserstein
metric, which we state in the next Lemma.

\begin{lemma}
  Let $k \geqslant 1$ and let $(X_i)_{i \geqslant 1}$ be a sequence of
  independent and identically distributed random variables taking
  values in $\Rr^k$. Assuming further that they admit a moment of
  order $2$, we denote by
$$
m = \expect(X_1) = \begin{pmatrix}
  \expect(X_{1,1}) \\
  \dots \\
  \expect(X_{1,k})
\end{pmatrix} 
$$
the mean value of $X_1$ and by
$\Sigma = (\sigma_{i,j})_{1 \leqslant i,j \leqslant k}$ the
covariance matrix of $X_1$, meaning that for all
$i,j \in \{1, \dots, k\}$,
$\sigma_{i,j} =\expect((X_{1,i} - \expect(X_{1,i}))(X_{1,j} -
\expect(X_{1,j})))$.  Then if $\mu_n$ denotes the law of
$\frac{(X_1 + \cdots X_n) - nm}{\sqrt{n}}$ we have
$$
\wass_1(\mu_n, \mcN(0, \Sigma)) \underset{n \to
  \infty}{\longrightarrow} 0.
$$
\end{lemma}

\begin{proof}
  Thanks to \cite[Th. 7.12]{villani}, convergence with respect to the
  $p$-Wasserstein metric is equivalent to the weak convergence of
  measures and the convergence of absolute moments of order
  $p$. However, in the setting of the Central Limit Theorem we have (denoting by $N$ a
  random variable with distribution $\mcN(0, \Sigma)$):
  $$
  \expect\Bigl(\Bigl| \frac{(X_1 + \cdots X_n) - nm}{\sqrt{n}}
  \Bigr|^2\Bigr) = \expect{|N|^2} = \mathrm{Tr}(\Sigma),
  $$
  hence the convergence of absolute moments of order $2$ is
  automatically satisfied. Therefore, the usual Central Limit Theorem
  (see, e.g.,~\cite[Th. 29.5]{billingsley}) which states the weak
  convergence of $\mu_n$ to $\mcN(0, \Sigma)$ immediately gives the
  apparently stronger statement
  $$
  \wass_2(\mu_n, \mcN(0, \Sigma)) \underset{n \to \infty}{\longrightarrow}
  0.
  $$

  Since $\wass_1 \leqslant \wass_2$, the conclusion follows.
\end{proof}

\begin{rem}
  We stated a qualitative result that suffices for our application,
  but there are several articles investigating the rate of convergence
  in the Central Limit Theorem with respect to Wasserstein
  metrics. Under the assumption that $\expect(|X_1|^4)$ is finite,
  \cite[Th.\,1]{bonis20} states that
  $\wass_2(\mu_n, \mcN(0, \Sigma)) \ll_k \frac{1}{\sqrt{n}}\cdot$ The
  dependence of the implicit constant with respect to the dimension
  $k$ is a more subtle question, we refer to \cite{bonis24} and the
  references therein for a recent account.
\end{rem}

\begin{proof}[Proof of Theorem $\ref{th: growing subgroups}$]
  Let $q$ be an odd prime, and $d = d(q)$ a prime divisor of
  $q-1$. By the triangle inequality for the metric $\wass_1$ we have
  $$
  \wass_1(\mu_{q,d}, \mcN(0, \Sigma)) \leqslant \wass_1(\mu_{q,d},
  \gamma_d) + \wass_1(\gamma_d, \mcN(0, \Sigma)).
  $$
  The second term converges to zero when $d$ tends to infinity thanks
  to the central limit theorem (the term
  $(\sqrt{d}Z_1 \dots Z_{d-1})^{-1}$ does not cause any issue because
  it has modulus $1/ \sqrt{d}$ so it converges almost surely to
  zero). Moreover, thanks to Lemma \ref{lem: quantitative duke garcia
    lutz} the first term is upper bounded by
  $ 4\sqrt{3} \sqrt{d} (d+1) q^{-\frac{1}{d-1}}$, so it suffices to
  show that the condition
  $$
  d \underset{q \to +\infty}{=} o \left( \frac{\log q}{\log \log
      q}\right)
  $$
  implies that this upper bound converges to zero as $q$ goes to
  infinity. This comes from the fact that Lambert $W_0$ function
  (which is the inverse bijection to $x \mapsto xe^x$ on
  $\left[- \frac{1}{e}, + \infty\right[$) satisfies
  $W_0(x) \sim_{x \to +\infty} \log(x)$, so the condition above may be
  rewritten as
  $$
  d  \underset{q \to \infty}{=} o \left( \frac{\log q}{W_0(\log q)}\right) \cdot
  $$
  Therefore, $d = \frac{\log q}{W_0(\log q)} \varepsilon(q)$ for some function $\varepsilon$ that tends to zero as $q$ tends to infinity, so
  $$
  d \log d = \frac{\log q}{W_0(\log q)} \varepsilon(q)
  \log\left(\frac{\log q}{W_0(\log q)} \varepsilon(q)\right)$$ and for
  $q$ large enough this is upper bounded by
  $$
  \frac{\log q}{W_0(\log q)} \varepsilon(q) \log\left(\frac{\log q}{W_0(\log q)} \right),
  $$
  but by definition of $W_0$ this is equal to
  $\log(q) \varepsilon(q)$. This shows that
  $d \log(d) \underset{q \to \infty}{=} o(\log q)$ and elementary
  manipulations show that this implies
  $d^{3/2} \underset{q \to \infty}{=} o
  \left(q^{\frac{1}{d-1}}\right)$, concluding the proof.
\end{proof}

Another type of factorization of $d$ for which the limiting
distribution can be determined is when $d$ is a power of fixed
prime. Indeed, when $d$ is of the form $r^b$ where $r$ is a prime
number and $b \geqslant 1$, the Laurent polynomial $g_d$ of
\cite{gaussian_periods,menagerie} can be made more explicit (this
comes from the fact that the coefficients of the cyclotomic polynomial
$\Phi_{r^b}$ are known). Precisely, \cite[Cor. 1]{menagerie} states
that the sums $S_d(a,q)$ become equidistributed with respect to the
pushforward measure of the Haar measure on $(\Ss^1)^{\varphi(r^b)}$
with respect to the Laurent polynomial $g_{r^b}$ defined by
$$
g_{r^b}\left(z_1, z_2, \ldots, z_{\varphi(r^b)}\right)=\sum_{j=1}^{\varphi(r^b)} z_j+\sum_{m=1}^{r^{b-1}} \prod_{\ell=0}^{r-2} z_{m+\ell r^{b-1}}^{-1}.
$$
Rearranging the terms according to their residue classes modulo
$r^{b-1}$ we can rephrase that statement as follows: the sums
$\frac{1}{\sqrt{d}} S_d(q,a)$ become equidistributed with respect to a
measure $\gamma_d$ which the law of a random variable
\begin{equation} \label{eq: sum Zij} \frac{1}{r^{b/2}}\sum_{i =
    1}^{r^{b-1}} Z_{i,1} + \cdots + Z_{i,r-1} + \frac{1}{Z_{i,1}\dots
    Z_{i, r-1}} \cdot
\end{equation}
where
$(Z_{i,j})_{1 \leqslant i \leqslant r^{b-1}, \ 1 \leqslant j \leqslant
  r-1}$ is a family of independent and identically distributed
Steinhaus random variables. By the same arguments as in the proof of Lemma \ref{lem:
  quantitative duke garcia lutz} (we only used the fact that $d$ is
prime to replace $[\Qq(\mmu_d) : \Qq]$ by $d-1$ and to make $\gamma_d$ more explicit, but the lemma holds for
arbitrary $d$, except for the description of the Laurent polynomial
$g_d$), we have
$$
\wass_1(\mu_{q,d}, \gamma_d) \leqslant 4\sqrt{3} \sqrt{d} (d+1) q^{-\frac{1}{\varphi(d)}}.
$$

This upper bound converges to zero as $d$ and $q$ tend to infinity
provided $d = o \left( \frac{\log q}{\log \log q}\right)$. Now, if $r$
is fixed and only $b$ varies, the sum \eqref{eq: sum Zij} may be
rewritten as
$$
\frac{1}{\sqrt{r}} \left(\frac{1}{\sqrt{r^{b-1}}} \sum_{i =
    1}^{r^{b-1}} X_i\right)
$$
where the
$X_i = Z_{i,1} + \cdots + Z_{i,r-1} + \frac{1}{Z_{i,1}\dots Z_{i,
    r-1}}$ are independent and identically distributed random
variables which have mean $0$. Thanks to the Central Limit Theorem, we
have
$$
\frac{1}{\sqrt{r^{b-1}}} \sum_{i = 1}^{r^{b-1}}
X_i\overset{\text{law}}{\longrightarrow} \mcN(0, \Sigma)
$$
where $\Sigma = \frac{r}{2} \mathrm{Id}$ is the covariance matrix of
$X_1$ (viewed as a random variable with values in $\Rr^2$). Taking
into account the factor $1/ \sqrt{r}$ in front of the sum, we obtain
the following result:

\begin{theorem}
  Let $r$ be a fixed prime. For all integers $b$ and all prime numbers
  $q$ such that $d = r^b$ divides $q-1$, we define the sums
  $\frac{1}{\sqrt{d}} S_d(q,a)$ as above. Then as $d$ and $q$ both
  tend to infinity with
  $d = o\left( \frac{\log q}{\log \log q}\right)$, they become
  equidistributed in the complex plane with respect to the normal
  distribution $\mcN(0, \frac{1}{2}\mathrm{Id})$.
\end{theorem}

\section{Equidistribution theorems for trace
  functions}\label{sec-deligne}

The Chebotarev Density Theorem and Deligne's Equidistribution Theorem
are prototypes of some fundamental results in number theory, and play
especially important roles as bridges between analytic and algebraic
aspects. Especially in the setting of finite fields, where Deligne's
Riemann Hypothesis is available, these statements are proved using
direct powerful estimates for Weyl sums, in the context of compact Lie
groups, which are usually non-abelian. They are therefore natural
targets for quantitative versions.

There are already results of this type in the literature, among which
we mention the paper~\cite{niederreiter} of Niederreiter (for
Kloosterman sums), the papers~\cite{fouvry-michel1}
and~\cite{fouvry-michel2} of Fouvry and Michel (for more general
exponential sums, including almost prime moduli), the paper~\cite{xi}
of Xi concerning Jacobi sums, and the recent work of Fu, Lau and
Xi~\cite{fulauxi} which provides a general statement for a suitable
discrepancy.  We will derive quite quickly a basic result using the
analogue for compact connected Lie groups of the Bobkov--Ledoux
theorem, due to Borda~\cite[Th.\,1]{borda_berry_esseen}.

\subsection{Borda's inequality}

In this section, we fix a compact connected Lie group~$K$, given with a
faithful finite-dimensional (unitary) representation, which we view as
an inclusion $K\subset \Un_m(\Cc)$ for some integer~$m\geq 1$.

We equip~$K$ with a Riemannnian metric induced by an Ad-invariant
positive definite bilinear form on $\mathrm{Lie}(K)$, and we assume that
this metric is normalized so that the associated Riemannian volume
measure is the probability Haar measure~$\mu_K$ on~$K$.
We denote by $\widehat{K}$ a set of representatives of irreducible
unitary representations of~$K$ (which can be identified with the set
of the characters of these representations). The structure theory of
compact Lie groups shows that there exists an integer~$r\geq 0$, the
rank of~$K$, such that these representations are parameterized by
vectors $\lambda$ in a cone~$\Lambda_K^+$ in an $r$-dimensional free
abelian group~$\Lambda_K$ (see,
e.g.,~\cite[p.\,67,\,Th.\,1]{lie9}). We denote by~$\rho_{\lambda}$ the
irreducible representation associated to~$\lambda$, and we usually
identify $\widehat{K}$ and~$\Lambda_K^+$. Further, we denote by
$\kappa(\lambda)$ the Casimir--Laplace eigenvalue of $\rho_{\lambda}$.

Given a probability measure~$\mu$ on~$K$ which is invariant under
conjugation and $\lambda\in\widehat{K}$, we denote by
\[
  \widehat{\mu}(\lambda)=\int_K
  \overline{\Tr(\rho_{\lambda}(g))}d\mu(g)
\]
the $\lambda$-Fourier coefficient of~$\mu$.

Finally, we  denote by $\|\lambda\|$ the natural euclidian norm
on~$\Lambda_K\otimes\Rr$.


\begin{theorem}[Borda]\label{th-borda-conjugacy}
  With notation as above, let~$\mu$ and~$\nu$ be conjugacy-invariant
  probability measures on~$K$. Let~$T\geq 1$ be a parameter. We have
  \[
    \wass_1(\mu,\nu)\ll T^{-1}+\Bigl( \sum_{1\leq \|\lambda\|\leq
      T} \frac{1}{\kappa(\lambda)}
    |\widehat{\mu}(\lambda)-\widehat{\nu}(\lambda)|^2 \Bigr)^{1/2}
  \]
  where the sum is over vectors $\lambda\in\Lambda_K^+$ parametrizing
  irreducible unitary representations~$\rho_{\lambda}$ of~$K$ and the
  implied constant depends only on~$K$.
\end{theorem}

\begin{proof}
  Borda~\cite[Th.\,1]{borda_berry_esseen} proves that
  \[
    \wass_1(\mu,\nu)\ll T^{-1}+\Bigl( \sum_{1\leq \|\lambda\|\leq
      T} \frac{\dim(\rho_{\lambda})}{\kappa(\lambda)}
    \|\widetilde{\mu}(\lambda)-\widetilde{\nu}(\lambda)\|_{HS}^2
    \Bigr)^{1/2}
  \]
  where 
  \[
    \widetilde{\mu}(\lambda)=\int_K\rho_{\lambda}(g)^*d\mu(g)
  \]
  (which is an element of~$\End(\rho_{\lambda})$) is the Fourier
  coefficient of~$\mu$ at~$\lambda$,
  with~$u\mapsto\|u\|_{HS}=(\Tr(u^*u))^{1/2}$ denoting the
  Hilbert--Schmidt norm on~$\End(\rho_{\lambda})$.

  Let~$\lambda\in\widehat{K}$.  Since~$\mu$ is conjugacy-invariant, the
  linear map $\widetilde{\mu}(\lambda)$ is a $K$-endomorphism
  of~$\rho_{\lambda}$ (see, e.g.,~\cite[p.\,402]{ts345}) hence is a
  multiple of the identity by Schur's Lemma (see,
  e.g.,~\cite[p.\,386,\,prop.\,6]{ts345}). Taking the trace, we obtain
  the (standard) formula
  \[
    \widetilde{\mu}(\lambda)=\frac{1}{\dim(\rho_{\lambda})}
    \int_K\overline{\Tr(\rho_{\lambda}(g))}d\mu(g)=
    \frac{1}{\dim(\rho_{\lambda})}\widehat{\mu}(\lambda) \mathrm{Id},
  \]
  and similarly for~$\nu$. We deduce that
  \[
    \wass_1(\mu,\nu)\ll T^{-1}+\Bigl( \sum_{1\leq \|\lambda\|\leq T}
    \frac{1}{\dim(\rho_{\lambda})\kappa(\lambda)}
    |\widehat{\mu}(\lambda)-\widehat{\nu}(\lambda)|^2\|\mathrm{Id}\|_{HS}^2
    \Bigr)^{1/2},
  \]
  and we finally conclude since
  $\|\mathrm{Id}\|_{HS}^2=\Tr(\mathrm{Id}^*\mathrm{Id})=\dim(\rho_{\lambda})$.
\end{proof}

\begin{remark}
  (1) The restriction to~$K$ connected is quite natural from the
  technical point of view, but somewhat restrictive for applications
  like those we are going to describe, since it is quite possible that
  the compact group in these cases is not connected.  We hope to
  provide an extension of Borda's inequality to the general case in a
  future work.

  (2) It should be possible to obtain an estimate uniform in terms
  of~$K$, and we also hope to come back to this later.
\end{remark}

For the application to equidistribution, we will consider in particular
the conjugacy-invariant probability measures $\delta_{g^{\sharp}}$ which
are defined for a conjugacy class $g^{\sharp}\subset K$ of~$g$ by the
integration formula
\[
  \int_K f(x)d\delta_{g^{\sharp}}(x)=\int_Kf(xgx^{-1})d\mu_K(x)
\]
for $f\colon K\to\Cc$ continuous (in other words, this is the image
measure of $\mu_K$ under the map $x\mapsto xgx^{-1}$). We call this
measure the \emph{uniform measure on the conjugacy class of~$g$}.  It
has the property that if $K^{\sharp}$ denotes the space of conjugacy
classes of~$K$ (which is a compact space in our case), the pushforward
of~$\delta_{g^{\sharp}}$ by the projection $K\to K^{\sharp}$ is the
delta mass at the given conjugacy class.

\begin{example}
  If~$K$ is finite then one finds easily that
  \[
    \delta_{g^{\sharp}}(\{x\})=\begin{cases}
      0&\text{ if } x\notin g^{\sharp}\\
      \frac{1}{|g^{\sharp}|}&\text{ if } x\in g^{\sharp}.
    \end{cases}
  \]

  This justifies our terminology.
\end{example}

\subsection{Deligne's equidistribution theorem}

We start by considering Deligne's Theorem in the case of curves.

\begin{theorem}\label{th-deligne}
  Let~$\ell$ be a prime number.  Let
  $(k_i,X_i,\mathcal{F}_i)_{i\in I}$ be an infinite sequence of data
  consisting of a finite field $k_i$ of characteristic different
  from~$\ell$, a smooth connected quasi-projective curve $X_i/k_i$ and
  a lisse $\ell$-adic sheaf $\mathcal{F}_i$ of weight~$0$
  on~$X_i$. Suppose that the following conditions are satisfied:
  \begin{enumth}
  \item we have $X_i(k_i)\not=\emptyset$ and $|X_i(k_i)|\to +\infty$,
  \item the complexity $c(\mathcal{F}_i)$ in the sense
    of~\cite[Def.\,3.2]{sffk} is bounded for all~$i$,
  \item the arithmetic and geometric monodromy groups
    of~$\mathcal{F}_i$ are equal, are independent of~$i$, and are
    connected for all~$i$.
  \end{enumth}

  Let~$K$ be a connected compact Lie group isomorphic to a maximal
  compact subgroup of the common arithmetic and geometric monodromy
  groups of the sheaves $\mathcal{F}_i$.

  For~$i\in I$ and~$x\in X_i(k_i)$, let~$\theta(x)\in K^{\sharp}$
  denote the unitarized Frobenius conjugacy class of~$x$ acting
  on~$\mathcal{F}_i$, which is well-defined because of the condition
  that~$\mathcal{F}_i$ is of weight~$0$.

  Then the conjugacy-invariant probability measures on~$K$ defined by
  \[
    \mu_i=\frac{1}{|X_i(k_i)|}
    \sum_{x\in X_i(k_i)}\delta_{\theta(x)}
  \]
  satisfy
  \begin{equation}\label{eq-wass-conj-class}
    \wass_1(\mu_i,\mu_K)\ll |k_i|^{-1/\dim(K)}
  \end{equation}
  for all~$i$. Furthermore, we have
  \begin{equation}\label{eq-wass-exp-sums}
    \wass_1(\Tr_*\mu_i,\Tr_*\mu_K)\ll |k_i|^{-1/\dim(K)}.
  \end{equation}
  
  In both cases, the implied constants depends on~$K$ and on the bound
  for the complexity of~$\mathcal{F}_i$.
\end{theorem}

\begin{proof}
  We will apply Borda's inequality as given in
  Theorem~\ref{th-borda-conjugacy}. We first estimate the Fourier
  coefficients. For $\lambda\in \widehat{K}$ non-trivial, we have of
  course $\widehat{\mu}_K(\lambda)=0$. On the other hand, by
  definition, the formula
  \[
    \widehat{\mu}_i(\lambda)
    =\frac{1}{|X_i(k_i)|}
    \sum_{x\in X_i(k_i)}\Tr(\rho_{\lambda}(\theta(x)))
  \]
  holds for all~$i$.  By the usual application of the formalism of
  trace functions and the Riemann Hypothesis over finite fields (see,
  for instance, Katz's account in~\cite[Ch.\,3]{gkm} or the one in
  Katz--Sarnak~\cite[Th.\,9.2.6]{katz-sarnak}), we obtain
  \begin{equation}\label{eq-basic-weyl-bound}
    |\widehat{\mu}_i(\lambda)-\widehat{\mu}_K(\lambda)|=
    |\widehat{\mu}_i(\lambda)|\leq
    |k_i|^{-1/2}\sigma_c(X_{i,\bar{k}},\rho_{\lambda}(\mathcal{F}_i)),
  \end{equation}
  where
  \[
    \sigma_c(X_{i,\bar{k}},\rho_{\lambda}(\mathcal{F}_i))=
    \sum_j \dim H^j_c(X_{i,\bar{k}},\rho_{\lambda}(\mathcal{F}_i))
  \]
  is the sum of Betti numbers of the sheaf associated to
  $\rho_{\lambda}$ and~$\mathcal{F}_i$.
  
  Since~$X_i$ is a curve, the upper bound
  \[
    \sigma_c(X_{i,\bar{k}},\rho_{\lambda}(\mathcal{F}_i))\ll
    \dim(\rho_{\lambda}) c(\mathcal{F}_i)
  \]
  holds by~\cite[Rem.\,6.34]{sffk}.

  For any value of the parameter~$T\geq 1$,
  Theorem~\ref{th-borda-conjugacy} then gives
  \[
    \wass_1(\mu_i,\mu_K)\ll
    \frac{1}{T}+\frac{c(\mathcal{F}_i)}{|k_i|^{1/2}} \Bigl(\sum_{1\leq
      \|\lambda\|\leq T}\frac{\dim(\rho_{\lambda})^2}{\kappa(\lambda)}
    \Bigr)^{1/2}.
  \]

  It is known that $\kappa(\lambda)\geq \|\lambda\|^2$ (see,
  e.g.,~\cite[p.\,77,\,Prop.\,4]{lie9}) and that
  $\dim(\rho_{\lambda})=O(\|\lambda\|^{(n-r)/2})$ (see,
  e.g.,~\cite[p.\,76,\,Cor.\,1]{lie9}), where $n$ is the dimension
  of~$K$ (and~$r$ is again its rank). Thus
  \[
    \sum_{1\leq \|\lambda\|\leq
      T}\frac{\dim(\rho_{\lambda})^2}{\kappa(\lambda)}= O\Bigl(
    \sum_{1\leq \|\lambda\|\leq T}\|\lambda\|^{n-r-2} \Bigr)=O(T^{n-2}),
  \]
  for $T\geq 1$, where the implied constant depends only on~$K$ (using
  simply the bound $O(T^r)$ for the number of lattice points in a ball
  of radius~$T$ in a lattice of rank~$r$). Hence
  \[
    \wass_1(\mu_i,\mu_K)\ll
    \frac{1}{T}+\frac{c(\mathcal{F}_i)}{|k_i|^{1/2}} T^{(n-2)/2}.
  \]

  Since we assumed that $c(\mathcal{F}_i)\ll 1$, picking
  $T=|k_i|^{1/n}$ leads to
  \[
    \wass_1(\mu_i,\mu_K)\ll |k_i|^{-1/n}=|k_i|^{-1/\dim(K)},
  \]
  as claimed.

  Since it is straightforward that the trace $K\to \Cc$ is Lipschitz,
  we deduce the bound
  $ \wass_1(\Tr_*\mu_i,\Tr_*\mu_K)\ll |k_i|^{-1/\dim(K)}$ by
  Theorem~\ref{th-wass}, (3).
\end{proof}

The formulation of the theorem may seem awkward, but it was chosen to
apply both in ``vertical'' and ``horizontal'' directions. We illustrate
these with some basic examples of each case. Before doing so, we point
out that~(\ref{eq-wass-exp-sums}) gives equidistribution for the
measures
\[
  \Tr_*\mu_i=\frac{1}{|X_i(k_i)|} \sum_{x\in
    X_i(k_i)}\delta_{\Tr(\theta(x))},
\]
which are typically the discrete measures associated to values of
families of exponential sums.

\begin{example}
  (1) [``Vertical'' equidistribution] In this setting, we consider a
  base finite field~$k$, algebraic curve~$X/k$ and lisse $\ell$-adic
  sheaf~$\mathcal{F}$ over~$X$, and we consider the family
  $(k_{n},X\times k_n,\mathcal{F}\otimes k_n)_{n\geq 1}$, where $k_n$
  denotes an extension of degree~$n$ of~$k$. The only assumption
  required to apply Theorem~\ref{th-deligne} in this case is that
  $\mathcal{F}$ has equal geometric and arithmetic monodromy groups, and
  that these are connected.

  This is the situation considered in the recent paper of Fu, Lau and Xi
  (see~\cite[Th.\,1.2]{fulauxi}), where the quality of equidistribution
  is expressed in terms of a discrepancy for certain boxes in the space
  of conjugacy classes of~$K$, identified with the quotient of a maximal
  torus by the Weyl group of~$K$. This work actually applies to an
  arbitrary base variety, and we provide a corresponding version of
  Theorem~\ref{th-deligne} in Theorem~\ref{th-deligne2}. Although the
  results are not directly comparable, we can observe that the bound
  they obtain for the discrepancy is of size $|k_n|^{-1/(2(|R^+|+1))}$,
  where $R^+$ is the set of positive roots of~$K$. The classical formula
  $\dim(K)=2|R^+|+r$ shows that our exponent is usually slightly worse,
  although we do obtain a bound of size $|k_i|^{-1/3}$ instead of
  $|k_i|^{-1/4}$ in the special case of $\SU_2(\Cc)$, which does occur
  in important cases such as classical Kloosterman sums or Birch sums.

  We spell out our result in the case of hyper-Kloosterman
  sums. Let~$r\geq 2$ be an integer. Let $\psi$ be a fixed non-trivial
  additive character of the base field~$k$, and let
  $\psi_n=\psi\circ \Tr_{k_n/k}$ be the corresponding character
  of~$k_n$. For $n\geq 1$ and $a\in k_n^{\times}$, let
  \[
    \Kl_r(a;k_n)=\frac{1}{|k_n|^{(r-1)/2}} \sum_{\substack{x_1,\ldots,
          x_r\in k_n^{\times}\\x_1\cdots x_r=a}}\psi_n(x_1+\cdots+x_r).
  \]

  Deligne constructed a lisse sheaf $\mathcal{K}_r$ on $X=\Gg_m/k$, pure
  of weight~$0$, with trace function equal to $a\mapsto \Kl_r(a;k)$, and
  Katz~\cite[Th.\,11.1]{gkm} proved that if the characteristic of~$k$ is
  odd, then this sheaf has geometric and monodromy groups isomorphic
  to~$\SL_r$ if~$r$ is odd, and~$\Sp_{r}$ if~$r$ is even. Thus the
  theorem applies with~$K=\SU_r(\Cc)$ if~$r$ is odd
  and~$K=\USp_{r}(\Cc)$ if~$r$ is even (in particular, with
  $K=\SU_2(\Cc)$ if~$r=2$).  Thus
  \[
    \wass_1\Bigl(\frac{1}{|k_n^{\times}|}\sum_{a\in k_n^{\times}}
    \delta_{\Kl_r(a;k_n)},\Tr_*\mu_K\Bigr) \ll |k_n|^{-1/\dim(K)},
  \]
  with
  \[
    \dim(K)=\begin{cases}
      \frac{r(r-1)}{2}&\text{ if $r$ is even},\\
      r^2-1&\text{ if $r$ is odd.}
    \end{cases}
  \]

  In the case~$r=2$, a specific feature of $\SU_2(\Cc)$ is that the
  trace map directly identifies the space of conjugacy classes with
  the interval $[-2,2]$, and the probability Haar measure with the
  Sato--Tate measure
  \[
    \frac{1}{\pi}\sqrt{1-\frac{x^2}{4}}dx
  \]
  on $[-2,2]$ (see, e.g.,~\cite[p.\,58,\,exemple]{lie9}).  Thus
  \[
    \wass_1\Bigl(\frac{1}{|k_n^{\times}|}\sum_{a\in k_n^{\times}}
    \delta_{\Kl_2(a;k_n)},\frac{1}{\pi}\sqrt{1-\frac{x^2}{4}}dx\Bigr)
    \ll |k_n|^{-1/3},
  \]
  where the Wasserstein distance is for measures on $[-2,2]$.

  (2) [``Horizontal equidistribution''] In this setting, we consider a
  smooth geometrically connected algebraic curve $X/\Zz[1/N]$ for some
  integer~$N\geq 1$, and the finite fields $\Ff_p$ for primes
  $p\to +\infty$ (possibly in a subsequence), and for each
  $p\not=\ell N$ we assume given a sheaf $\mathcal{F}_{p}$ on $X/\Ff_p$
  in such a way that the complexity of $\mathcal{F}_p$ remains bounded
  independently of~$p$, and the geometric and arithmetic monodromy
  groups of $\mathcal{F}_p$ are equal, connected, and independent
  of~$p$. Denoting by $\theta(x;p)$ the conjugacy class in~$K$
  corresponding to a point $x\in X(\Ff_p)$, we obtain from
  Theorem~\ref{th-deligne} the estimate
  \[
    \wass_1\Bigl(\frac{1}{|X(\Ff_p)|}\sum_{x\in X(\Ff_p)}
    \delta_{\Tr(\theta(x;p))},\mu_K\Bigr) \ll p^{-1/\dim(K)}
  \]
  for all~$p$ such that~$X(\Ff_p)$ is non-empty.
  
  This setting applies again to the Kloosterman sheaves over $k=\Ff_p$
  for $p$ varying, since these have bounded complexity (see,
  e.g.,~\cite[Prop.\,7.8,\,(1)]{sffk}). Thus, for~$r\geq 2$, we have
  the bound
  \[
    \wass_1\Bigl(\frac{1}{p-1}\sum_{a\in \Ff_p^{\times}}
    \delta_{\Kl_r(a;p)},\Tr_*\mu_K\Bigr) \ll p^{-1/\dim(K)},
  \]
  where~$K$ is either~$\SU_r(\Cc)$ or~$\USp_r(\Cc)$, as in the
  previous example.

  Here is a concrete application to a ``shrinking target'' problem for
  hyper-Kloosterman sums.

  \begin{proposition}\label{pr-hyper-quant}
    Let~$r\geq 3$ be an odd integer. Let~$g_0\in \SU_r(\Cc)$ be a matrix
    with at least three different eigenvalues.

    There exists a constant~$c>0$ and $p_0\geq 3$ such that, for any
    prime number~$p\geq p_0$, the estimate
    \[
      |\{a\in\Ff_p^{\times}\,\mid\, |\Kl_r(a;p)-\Tr(g_0)|\leq
      cp^{-\tfrac{1}{3(r^2-1)}}\}|\gg p^{1-\tfrac{2}{3(r^2-1)}}
    \]
    holds. In particular, for $p$ large enough, there exists
    $a\in\Ff_p^{\times}$ such that
    \[
      |\Kl_r(a;p)-\Tr(g_0)|\leq cp^{-\tfrac{1}{3(r^2-1)}}.
    \]
  \end{proposition}

  \begin{lemma}\label{lm-tube}
    Let $r\geq 3$ be an odd integer.  Let~$g_0\in \SU_r(\Cc)$. If $g_0$
    has at least three different eigenvalues, then for $\eps>0$
    sufficiently small, the set
    \[
      M_{\eps}=\{g\in \SU_r(\Cc)\,\mid\, |\Tr(g)-\Tr(g_0)|\leq\eps\}
    \]
    satisfies
    \[
      \mu_K(M_{\eps})\gg \eps^2.
    \]
  \end{lemma}

  \begin{proof}
    We denote~$K=\SU_r(\Cc)$ to simplify notation.  Let
    \[
      M=\{g\in K\,\mid\, \Tr(g)=\Tr(g_0)\},
    \]
    and let~$U\subset M$ be the subset where $g$ has at least three
    distinct eigenvalues.  We claim that~$U$ is a non-empty
    submanifold of~$K$ of codimension~$2$. Assuming this, we observe
    that $M_{\eps}$ contains a geodesic tube around the
    submanifold~$U$, and the lemma follows from the lower bounds
    \[
      \mu_K(M_{\eps})\geq \mu_K(U)\gg \eps^2,
    \]
    where the last inequality is obtained for instance
    from~\cite[Th.\,9.23]{gray} (the exponent~$2$ coincides with the
    codimension of~$U$ in~$K$; to apply this theorem, we use the fact
    that~$U$ is relatively compact).

    To prove the claim, we first note that $U$ is not empty since it
    contains~$g_0$ by assumption.  Thus by elementary differential
    geometry, it only remains to check that the trace from $K$ to~$\Cc$
    is a submersion at any point $g\in U$.  To see this, we identify
    $\Cc$ with $\Rr^2$ as an $\Rr$-vector space.  Since the trace is a
    smooth map from~$K$ to~$\Cc=\Rr^2$, we need to prove that the
    differential of the trace at~$g$ has rank~$2$ as an $\Rr$-linear map
    from the tangent space $T_gK$ to~$\Rr^2$.

    Since the trace is linear, this differential is
    identified with the trace from $T_gK$ to~$\Rr^2$.  The tangent
    space~$T_gK$ is identified with the space of matrices
    $g\mathrm{Lie}(K)$, hence the condition is that the map
    \[
      h\mapsto \Tr(gh)
    \]
    should have rank~$2$ on the Lie algebra of~$K$.

    Using conjugation invariance, it suffices to prove this condition
    when~$g$ is a diagonal matrix, say with diagonal coefficients
    $(\alpha_j)_{1\leq j\leq r}$.  In terms of the coefficients
    $(h_{j,k})$ of~$h$, we then have
    \[
      \Tr(gh)=\sum_{j=1}^r\alpha_j h_{j,j}.
    \]

    Recall that $\mathrm{Lie}(K)$ is the real vector space of
    $r\times r$ skew anti-symmetric matrices with trace~$0$. In
    particular, for $1\leq j<k\leq r$, the Lie algebra
    $\mathrm{Lie}(K)$ contains the diagonal matrix
    $d_{j,k} = (h_{a,b})_{1 \leq a,b \leq r}$ with all coefficients
    zero except $h_{j,j}=i$ and $h_{k,k}=-i$. We then have
    $\Tr(gd_{j,k})=i(\alpha_j-\alpha_k)$.

    If the rank of the map is $\leq 1$, it follows that these values
    span over~$\Rr$ a proper subspace of~$\Rr^2$. If~$g$ has at least
    three different eigenvalues $\alpha_{j_1}$, $\alpha_{j_2}$,
    $\alpha_{j_3}$, then it follows that they are aligned in the
    complex plane, which is however impossible since they lie on the
    unit circle.\footnote{\ Note that the converse is also true: if
      $g$ has only two (possibly equal) eigenvalues $\alpha$ and
      $\beta$, then the condition that $\Tr(h)=0$ implies that the
      image of the differential is $i(\alpha-\beta)\Rr$.}
  \end{proof}

  \begin{remark}
    If~$r$ is odd and if $\Tr(g_0)=0$, then we have in fact~$U=M$: if
    $g\in\SU_r(\Cc)$ has only eigenvalues $\alpha$, with multiplicity
    $s$, and $\beta$ with multiplicity $r-s$, then from
    $s\alpha+(r-s)\beta=0$ we deduce that $s=r-s$ (since
    $|\alpha|=|\beta|=1$), which is impossible.
  \end{remark}
  
  \begin{proof}[Proof of Proposition~\ref{pr-hyper-quant}]
    We write $K=\SU_r(\Cc)$.  From the previous lemma, we know that the
    subset
    \[
      M_{\eps}=\{g\in K\,\mid\, |\Tr(g)-\Tr(g_0)|\leq \eps\}
    \]
    has volume
    \[
      \mu_K(M_{\eps})\gg \eps^2
    \]
    for $\eps>0$ small enough.
    
    We now consider a prime~$p$ and a real number~$\eps>0$ for which the
    above fact holds. Define $\varphi\colon K\to\Cc$ by
    $\varphi(g)=\Phi(\Tr(g))$ where $\Phi$ is a function on~$\Cc$ which
    is supported on the disc of radius $2\eps$ around $\Tr(g_0)$, equal
    to~$1$ on the disc of radius $\eps$, and has Lipschitz constant
    $\ll \eps^{-1}$.

    We have
    \[
      |\{a\in\Ff_p^{\times}\,\mid\, |\Kl_r(a;p)-\Tr(g_0)|\leq
      2\eps\}|\geq (p-1)\int_{K}\varphi\, d\mu_p.
    \]

    Note that 
    \[
      \int_{K}\varphi\, d\mu_K \geq \mu_K(M_{\eps})\gg \eps^2,
    \]
    and therefore, by Theorem~\ref{th-deligne} combined with the
    Kantorovich--Rubinstein Theorem, we have
    \[
      \int_{K}\varphi\, d\mu_p\gg \mu_K(\eps)+
      O(\eps^{-1}p^{-1/\dim(K)})\gg \eps^2+O(\eps^{-1}p^{-1/\dim(K)}).
    \]

    Taking $\eps=cp^{-1/(3\dim(K))}$ for a suitably large
    constant~$c>0$, which is permitted if $p$ is large enough, this
    gives the stated result.
  \end{proof}

  We note that for matrices of trace zero, a comparable result was
  essentially proved by Fouvry and Michel~\cite{fouvry-michel1}, using
  the Weyl integration formula to estimate (in effect) the volume
  of~$M_{\eps}$.

  Also, one can handle the cases excluded in the statement in similar
  ways. For instance, suppose that~$g_0=\xi\mathrm{Id}$ is a scalar
  matrix, with~$\xi^r=1$. Then~$\Tr(g_0)=r\xi$, and no other matrix
  has the same trace, so that the set $T_{\eps}$ is a neighborhood
  of~$g_0$.  By~\cite[Th.\,9.23]{gray} again, it has measure
  $\sim \alpha \eps^{r^2-1}$ for some~$\alpha>0$, and we obtain by the
  same argument as above the estimate
  \[
    |\{a\in\Ff_p^{\times}\,\mid\, |\Kl_r(a;p)-r\xi|\leq
    cp^{-\tfrac{1}{r^2(r^2-1)}}\}|\gg p^{1-\tfrac{1}{r^2}}
  \]
  for suitable~$c>0$ and all $p$ large enough.

  The case~$r=2$ behaves slightly differently because the trace is
  then real-valued. But we get for instance the existence of a
  constant $c>0$ such that, for all $p$ large enough, there exists of
  $a\in\Ff_p^{\times}$ such that
  \[
    \Kl_2(a;p)\geq 2-cp^{-2/15}
  \]
  where the exponent arises because the Sato--Tate measure of the
  interval $[2-\eps,2]$ is of order of magnitude $\eps^{3/2}$ for
  $\eps\to 0$ (compare with the result of
  Niederreiter~\cite{niederreiter}).

\end{example}

The following variant of Theorem~\ref{th-deligne} is also of
independent interest: it relaxes the condition that $X_i$ is a curve,
but only applies in the vertical setting, as in~\cite{fulauxi}.

\begin{theorem}\label{th-deligne2}
  Let~$\ell$ be a prime number.  Let $X/k$ be a smooth geometrically
  connected algebraic variety over a finite field~$k$ of characteristic
  $\not=\ell$. Let $\mathcal{F}$ be a lisse $\ell$-adic sheaf of
  weight~$0$ on~$X$. Suppose that the arithmetic and geometric monodromy
  groups of~$\mathcal{F}$ are equal and connected.

  Let~$K$ be a connected compact Lie group isomorphic to a maximal
  compact subgroup of the common arithmetic and geometric monodromy
  groups of $\mathcal{F}$.

  Denoting by $k_i$ the extension of~$k$ of degree~$i$ in some algebraic
  closure of~$k$, the conjugacy-invariant measures
  \[
    \mu_i=\frac{1}{|X(k_i)|} \sum_{x\in X(k_i)}\delta_{\theta(x)},
  \]
  defined for integers~$i\geq 1$ such that $X(k_i)\not=\emptyset$,
  satisfy
  \begin{equation}\label{eq-deligne2}
    \wass_1(\mu_i,\mu_K)\ll |k_i|^{-1/\dim(K)}, \quad\quad
    \wass_1(\Tr_*\mu_i,\Tr_*\mu_K)\ll |k_i|^{-1/\dim(K)},
  \end{equation}
  where the implied constants depend on~$X$ and $\mathcal{F}$.
\end{theorem}

\begin{proof}
  We follow the argument in the proof of Theorem~\ref{th-deligne}, up
  to~(\ref{eq-basic-weyl-bound}). Then, we apply the Katz--Sarnak
  bound~\cite[Th.\,9.2.6,\,(3)]{katz-sarnak} for sums of Betti numbers
  namely
  \[
    \sigma_c(X_{\bar{k}},\rho_{\lambda}(\mathcal{F}))\ll \dim(\rho_{\lambda}),
  \]
  where the implied constant depends on $(X,\mathcal{F})$ (a possible
  value for this constant is given
  in~\cite[loc.\,cit.\,,\,(4)]{katz-sarnak}, namely the sum of Betti
  numbers of an auxiliary finite étale covering $X'\to X$ which
  trivializes the reduction of a model of $\mathcal{F}$ over a finite
  extension of $\Zz_{\ell}$).
\end{proof}

\begin{remark}
  To illustrate the strength of Theorems~\ref{th-deligne}
  and~\ref{th-deligne2}, consider the vertical situation for a variety
  $X/k$, and assume that some ``oracle'' presented~(\ref{eq-deligne2})
  as true, without indications on the proof, in the form
  \[
    \wass_1(\mu_i,\mu_K)\ll |k_i|^{-\alpha}
  \]
  for some $\alpha>0$. How strong is such an inequality?

  It is in fact quite powerful. Indeed, let $\lambda\in \widehat{K}$ be
  a non-trivial representation. It is elementary that the character of
  $\rho_{\lambda}$ is a Liptschitz map from~$K$ to~$\Cc$, and hence the
  bound~(\ref{eq-wass-conj-class}), together with
  Kantorovich--Rubinstein duality, implies \emph{a priori} an estimate
  of the form
  \[
    \frac{1}{|X(k_i)|}\sum_{x\in X(k_i)} \Tr\rho_{\lambda}(\theta(x))=
    \widehat{\mu}_i(\lambda)\ll |k_i|^{-\alpha}
  \]
  for $i\geq 1$, where the implied constant depends on~$\lambda$. This
  estimate, even if it is weaker than the Riemann Hypothesis (because
  $\alpha<1/2$), it amounts concretely to a \emph{zero-free strip} (of
  width $\alpha$) for the associated $L$-function
  \[
    \exp\Bigl(\sum_{i\geq 1}\frac{1}{i}\Bigl(\sum_{x\in X(k_i)}
    \Tr(\rho_{\lambda}(\theta(x)))\Bigr)\frac{1}{|k|^{is}}\Bigr).
  \]

  Moreover, this can be made more uniform. It is elementary (by
  decomposing the character, viewed as a function on a maximal torus of
  the group~$K$, as a sum of characters, and differentiating, observing
  that the highest weight gives ``maximal oscillations'') that the
  character $\Tr\rho_{\lambda}$ is $c$-Lipschitz with
  \[
    c=O(\|\lambda\|\dim(\rho_{\lambda}))
  \]
  for some absolute constant (and in fact, as was pointed out by
  P. Nelson, this estimate is sharp in terms of $\lambda$, as follows
  from a general result~\cite[Th.\,1]{shi-xu} of Shi and Xu, for
  instance), so that the Wasserstein bound implies
  \[
    \frac{1}{|X(k_i)|}\sum_{x\in X(k_i)} \Tr\rho_{\lambda}(\theta(x))
    \ll \|\lambda\|\dim(\rho_{\lambda})|k_i|^{-\alpha}
  \]
  for all $\lambda$, which is also a pretty good bound in this respect.
\end{remark}

\subsection{Equidistribution theorems for arithmetic Fourier
  transforms}

In~\cite{mellin}, Katz generalized Deligne's equidistribution theorem
to families of exponential sums parameterized by a multiplicative
character of a finite field. We can easily obtain a quantitative
version similar to Theorem~\ref{th-deligne} in this setting. We assume
here some familiarity with the notation and terminology of Katz, but
we will give concrete examples after the proof.

\begin{theorem}\label{th-katz}
  Let~$\ell$ be a prime number.  Let $(k_i,\mathcal{F}_i)_{i\in I}$ be
  an infinite sequence of data consisting of a finite field $k_i$ of
  characteristic different from~$\ell$ and a constructible $\ell$-adic
  sheaf $\mathcal{F}_i$ on~$\Gg_m/k_i$ with no Kummer sheaf as
  subsheaf or quotient. Suppose that the following conditions are
  satisfied:
  \begin{enumth}
  \item we have $|k_i|\to +\infty$,
  \item the complexity $c(\mathcal{F}_i)$ in the sense of~\cite{sffk}
    is bounded for all~$i$,
  \item the object $\mathcal{F}_i(1/2)[1]$ of Katz's category
    $\mathcal{P}_{arith}$ has weight~$0$ for all~$i$,
  \item the arithmetic and geometric tannakian monodromy groups of the
    objects $\mathcal{F}_i(1/2)[1]$ are equal, are independent of~$i$,
    and are connected for all~$i$.
  \end{enumth}

  Let~$K$ be a connected compact Lie group isomorphic to a maximal
  compact subgroup of the common arithmetic and geometric tannakian
  monodromy groups of the objects $\mathcal{F}_i(1/2)[1]$.

  For~$i\in I$, let $\mathcal{X}_i$ be the set of
  characters~$\chi\colon k_i^{\times}\to \Cc^{\times}$ which are
  ``good'' for~$\mathcal{F}_i(1/2)[1]$, and
  for~$\chi\in \mathcal{X}_i$, let~$\theta(\chi)\in K^{\sharp}$ denote
  the unitarized Frobenius conjugacy class associated to~$\chi$
  and~$\mathcal{F}_i$.

  Then the conjugacy-invariant measures
  \[
    \mu_i=\frac{1}{|\mathcal{X}_i|} \sum_{\chi \in
      \mathcal{X}_i}\delta_{\theta(\chi)}
  \]
  satisfy
  \begin{equation}\label{eq-wass-katz}
    \wass_1(\mu_i,\mu_K)\ll \frac{1}{\log |k_i|},\quad\quad
    \wass_1(\Tr_*\mu_i,\Tr_*\mu_K)\ll \frac{1}{\log |k_i|}
  \end{equation}
  where the implied constants depends on~$K$ and on the bound for the
  complexity of~$\mathcal{F}_i$.
\end{theorem}

\begin{proof}
  This follows essentially simply by adapting the proof of
  Theorem~\ref{th-deligne} to the situation of discrete Mellin
  transforms, as explained in~\cite[Ch.\,28]{mellin}. In particular,
  combining the proof of~\cite[Th.\,28.1]{mellin}
  and~\cite[Prop.\,6.36]{sffk}, we deduce that there exists a
  constant~$c>0$, independent of~$i$, such that
  \[
    |\widehat{\mu}_i(\lambda)|\ll
    \frac{\|\lambda\|c^{\|\lambda\|}}{|k_i|^{1/2}}
  \]
  for all non-trivial~$\lambda\in\widehat{K}$. For $T\geq 1$, it
  follows that
  \[
    \wass_1(\mu_i,\mu_K)\ll \frac{1}{T}+ \frac{1}{|k_i|^{1/2}}\Bigl(
    \sum_{1\leq \|\lambda\|\leq
      T}\frac{\|\lambda\|^2}{\kappa(\lambda)}c^{2\|\lambda\|}\Bigl)^{1/2}
    \ll \frac{1}{T}+\frac{c_1^{T}}{\sqrt{|k_i|}}
  \]
  for any constant~$c_1>c$. Taking $T=\alpha\log(|k_i|)$ for suitably
  small~$\alpha$ gives
  \[
    \wass_1(\mu_i,\mu_K)\ll \frac{1}{\log|k_i|}
  \]
  (which could be very slightly improved using the optimal choice
  of~$T$, based on the Lambert function).
\end{proof}

\begin{example}
  For $p$ prime $\not=2$ and $\chi$ a non-trivial multiplicative
  character of~$\Ff_p$, consider the sums
  \[
    R(\chi;p)=\frac{1}{\sqrt{p}}\sum_{x\in
      \Ff_p^{\times}\setminus\{1\}}
    \chi(t)\psi\Bigl(\frac{x+1}{x-1}\Bigr)
  \]
  where~$\psi(x)=e(x/p)$ is the additive character modulo~$p$, which
  arose in work of Kurlberg, Rosenzweig and Rudnick~\cite{krr}
  (related to discrete arithmetic quantum chaos). For any odd prime,
  Katz showed that there is a sheaf $\mathcal{F}_p$ on~$\Gg_m/\Ff_p$
  for which $\mathcal{X}_p$ is the set of non-trivial characters, and
  $\Tr(\theta(\chi))=R(\chi;p)$ for each such character. Katz further
  showed~\cite[Th.\,14.5]{mellin} that the corresponding arithmetic
  and geometric tannakian groups are equal to $\SL_2$, and thus we can
  apply Theorem~\ref{th-katz} with $K=\SU_2(\Cc)$ to deduce the bound
  \[
    \wass_1\Bigl(\frac{1}{p-2}\sum_{\chi\not=1}
    \delta_{R(\chi;p)},\frac{1}{\pi}\sqrt{1-\frac{x^2}{4}}dx\Bigr) \ll
    \frac{1}{\log p}.
  \]
\end{example}

\subsection{Final remarks}

It is not clear if the logarithmic convergence rate in
Theorem~\ref{th-katz}, in comparison with that of
Theorem~\ref{th-deligne}, is an artefact of the method or a genuine
phenomenon.  It boils down to the fact that the estimate for the
complexity (or sums of Betti numbers) in the case of the Mellin
transform depends exponentially on the ``complexity'' of $\lambda$,
whereas we have a linear bound for sheaves on curves, or in the
vertical situation of Theorem~\ref{th-deligne2}.  Since the bound is
proved by estimating the complexity of a much ``larger'' object
containing the object $\rho_{\lambda}(\mathcal{F}_i)(1/2)[1]$ (in the
tannakian sense), it would seem natural to expect that a linear bound
should also hold in this case, and it would be extremely interesting
to improve the complexity bounds in this situation (or to prove that,
in fact, it cannot be improved in general).

Another remark is that there are a number of other variants of the
previous theorems which could be included, and which we only mention for
completeness (giving details would be very repetitive).

\begin{enumerate}
\item One can obtain a version of Theorem~\ref{th-deligne} for arbitrary
  base varieties (not necessarily curves) using~\cite[Prop.\,6.33]{sffk}
  for the Betti number estimates; the estimate here would also be at
  best logarithmic, for similar reasons to what happens in
  Theorem~\ref{th-katz}.
\item Similarly, one could obtain a version of Theorem~\ref{th-katz} for
  general arithmetic Fourier transforms on connected commutative
  algebraic group, as in the work of Forey, Fresán and
  Kowalski~\cite[Ch.\,4]{ffk}. This would have the same issue with the
  Betti number bounds, and would currently be restricted to the vertical
  direction in general, due to the current status of the basic results
  in~\cite{ffk}.
\end{enumerate}

\section*{Appendix: proof of the Bobkov--Ledoux inequality}
\label{sec: appendix proof}

In this section, we reproduce and combine the arguments of Bobkov and
Ledoux in \cite{bl_truncated,bl_gaussian} to obtain the variant of
\cite[Eq. (1.6)]{bl_truncated} stated in Theorem~\ref{th-wass},
(7). Following the original source, for an integer~$d\geq 1$, we
identify here the torus $(\Rr/\Zz)^d$ with $Q_d=[0,2\pi[^d\subset\Rr^d$
with the distance~$\rho_d$ defined in~(\ref{eq-rhod}).

For a random vector $X = (X_1, \dots, X_d)$ in $\Rr^d$, we denote its
characteristic function by $\varphi_X$, and we recall that it is
defined for all $s = (s_1, \dots, s_d) \in \Rr^d$ by
$\varphi_X(s) = \expect\left(e^{i X\cdot s}\right)$ where $X \cdot s$
denotes the usual dot product on $\Rr^d$.

We say that a map $u \colon \Rr^d \to \Rr$ is $2 \pi$-periodic if for
all $m \in \Zz^d$, for all $x \in \Rr^d$, $u(x+ 2 \pi m) = u(x)$.

\begin{lemma}\label{lem: lip1 smooth}
  Denote by $\mathrm{Lip}^{2\pi}_1(\Rr^d, \Rr)$ the set of maps
  $v \colon \Rr^d \to \Rr$ that are $2\pi$-periodic and $1$-Lipschitz
  (with respect to the euclidean norm on $\Rr^d$). Then we have that
  for all Borel probability measures $\mu$ and $\nu$ on
  $(Q^d, \rho_d)$,
  $$
  \sup_{v \in \mathrm{Lip}^{2\pi}_1(\Rr^d, \Rr)} \Bigl|\int_{Q^d} v
    d \mu - \int_{Q^d} v d \nu \Bigr| = \sup_{w \in
    \mathrm{Lip}^{2\pi}_1(\Rr^d, \Rr) \cap \mathcal{C}^{\infty}}
  \Bigl|\int_{Q^d} w d \mu - \int_{Q^d} w d \nu
  \Bigr|.
  $$
\end{lemma}

\begin{proof}
  This is a standard smoothing argument by convolution. If
  $v \in \mathrm{Lip}^{2\pi}_1(\Rr^d, \Rr)$ then one easily checks
  that for all $\varepsilon > 0$,
  $$
  v_{\varepsilon}(x) = \frac{1}{(2 \pi \varepsilon^2)^{d/2}} \int_{\Rr^d} v(x-y) e^{-\frac{|y|^2}{2 \varepsilon^2}} dy
  $$
  defines a function in
  $\mathrm{Lip}^{2\pi}_1(\Rr^d, \Rr) \cap \mathcal{C}^{\infty}$ that
  satisfies
  $\| v_{\varepsilon} - v \|_{\infty, \Rr^d} \underset{\varepsilon \to
    0}{\longrightarrow} 0$. Then the result follows from this
  approximation by a smooth function.
\end{proof}

\begin{lemma}\label{lem: cauchy-schwarz}
  Let $\mu$ and $\nu$ be two Borel probability measures on
  $(Q^d, \rho_d)$. The following inequality holds:
  $$
  \sup_{w \in \mathrm{Lip}^{2\pi}_1(\Rr^d, \Rr) \cap
    \mathcal{C}^{\infty}} \Bigl|\int_{Q^d} w d \mu -
  \int_{Q^d} w d \nu \Bigr| \leqslant \Bigl(\sum_{m \in \Zz^d
    \setminus \{0\}}^{} \frac{|\hat{\mu}(m) -
    \hat{\nu}(m)|^2}{|m|^2}\Bigr)^{1/2} \cdot
  $$
\end{lemma}

\begin{proof}
  Let
  $w \in \mathrm{Lip}^{2\pi}_1(\Rr^d, \Rr) \cap
  \mathcal{C}^{\infty}$. Since $w$ is smooth, it admits a Fourier
  series expansion that converges absolutely:
  $$
  w(x)= \sum_{m \in \Zz^d}^{} a_m e^{i m \cdot x}.
  $$
  Moreover, one can differentiate term by term, so that for all $k \in \{1, \dots, d\}$, 
  $$
  \frac{\partial w}{\partial x_k} (x) = \sum_{m \in \Zz^d}^{} i m_{k} a_m  e^{i m \cdot x}.
  $$
  Then thanks to Parseval's equality (the $L^2$-norm of
  $ \frac{\partial w}{\partial x_\ell} $ equals the $\ell^2$-norm of
  its sequence of Fourier coefficients):
  $$
  \frac{1}{(2 \pi)^d} \int_{Q^d}^{} \Bigl| \frac{\partial w}{\partial
    x_k} (x)\Bigr|^2 dx = \sum_{m \in \Zz^d}^{} |m_k|^2
  |a_m|^2.
  $$
  Summing over $k \in \{1, \dots, d\}$ yields
  $$
  \frac{1}{(2 \pi)^d} \int_{Q^d}^{} \Bigl| \nabla w(x)\Bigr|^2
  dx = \sum_{m \in \Zz^d}^{} |m|^2 |a_m|^2.
  $$
  Finally, we use the fact that $w$ is $1$-Lipschitz to deduce that
  the norm of its gradient is always bounded above by $1$. Therefore,
  \begin{equation} \label{eq: less than 1}
    \sum_{m \in \Zz^d}^{} |m|^2 |a_m|^2 \leqslant 1.
  \end{equation}
  To conclude, we first write
  \begin{align*}
  \Bigl|\int_{Q^d} w d \mu - \int_{Q^d} w d \nu
  \Bigr|^2 &= \Bigl| \sum_{m \in \Zz^d} a_m \Bigl( \int_{Q^d} e^{i m
    \cdot x} d \mu (x)- \int_{Q^d} e^{i m \cdot x} d
             \nu(x) \Bigr) \Bigr|^2\\
    &= \Bigl| \sum_{m \in \Zz^d} a_m \Bigl(
  \hat{\mu}(m) - \hat{\nu}(m) \Bigr) \Bigr|^2
  \end{align*}
  then we observe that the two Fourier coefficients coincide at
  $m = 0$, so the right-hand side may be rewritten as
  $$
  \Bigl| \sum_{m \in \Zz^d\setminus \{0\}} |m|a_m \Bigl( \frac{
    \hat{\mu}(m) - \hat{\nu}(m) }{|m|} \Bigr) \Bigr|^2
  $$
  and the conclusion follows from the Cauchy--Schwarz inequality and
  \eqref{eq: less than 1}.
\end{proof}

One deduces quickly the following corollary:

\begin{corollary}\label{lem: without smoothing}
  Let $\mu, \nu$ be two Borel probability measures on $(Q^d,
  \rho_d)$. The following inequality holds:
  $$
  \wass_1(\mu,\nu) \leqslant \Bigl(\sum_{m \in \Zz^d \setminus
    \{0\}}^{} \frac{|\hat{\mu}(m) -
    \hat{\nu}(m)|^2}{|m|^2}\Bigr)^{1/2}.
  $$
\end{corollary}

\begin{proof}
  Thanks to the dual formulation of Th. \ref{th-wass} (6), we have
  $$
  \wass_1(\mu,\nu) = \sup_{u \in \mathrm{Lip}_1(Q^d, \Rr)}\Bigl|
  \int_{Q^d} u d \mu - \int_{Q^d} u d \nu \Bigr|.
  $$
  
  Now, if $u \colon (Q^d, \rho_d) \to \Rr$ is a $1$-Lipschitz map
  (with respect to $\rho_d$) then its $2\pi$-periodic extension to
  $\Rr^d$ is a $1$-Lipschitz map with respect to the euclidean norm on
  $\Rr^d$. Conversely, any $v \in \mathrm{Lip}^{2\pi}_1(\Rr^d, \Rr)$
  satisfies that $v_{|Q^d}$ is $1$-Lipschitz with respect to $\rho_d$.
  Therefore,
  $$
  \sup_{u \in \mathrm{Lip}_1(Q^d, \Rr)}\Bigl| \int_{Q^d} u d
  \mu - \int_{Q^d} u d \nu \Bigr| = \sup_{v \in
    \mathrm{Lip}^{2\pi}_1(\Rr^d, \Rr)} \Bigl|\int_{Q^d} v d
  \mu - \int_{Q^d} v d \nu \Bigr|
  $$
  and thanks to Lemma \ref{lem: lip1 smooth} and \ref{lem:
    cauchy-schwarz}, the right-hand side is bounded by
  $$
  \Bigl(\sum_{m \in \Zz^d \setminus \{0\}}^{} \frac{|\hat{\mu}(m) -
    \hat{\nu}(m)|^2}{|m|^2}\Bigr)^{1/2}.
  $$
\end{proof}

This corollary is not suitable for all applications because the series
on the right-hand side may well diverge. In order to obtain a more
useful inequality, the method is the usual application of convolution
by a measure whose sequence of Fourier coefficients is compactly
supported. Recall that if $\mu$ and $\nu$ are Borel probability
measures on $(\Rr/\Zz)^d$, the convolution $\mu * \nu$ is the unique
Borel probability measure on $(\Rr/\Zz)^d$ such that
$$
\int_{(\Rr/\Zz)^d} u d(\mu * \nu) = \int_{(\Rr/\Zz)^{2d}}
u(x+y) d \mu(x) d\nu(y)
$$
for all continuous functions $u \colon (\Rr/\Zz)^d \to \Cc$. As usual,
one can identify $\mu$, $\nu$ and $\mu*\nu$ with measures on~$Q^d$,
and then
$$
\int_{\Rr^d} u d(\mu * \nu) = \int_{\Rr^d\times\Rr^d} u(x+y)
d \mu(x) d\nu(y)
$$
for all functions $u\colon \Rr^d\to\Cc$ which are continuous and
$2 \pi$-periodic.


\begin{lemma}\label{lem: regularization cost}
  Let $\mu$ and $\nu$ be two Borel probability measures on
  $(Q^d, \rho_d)$ and let $H = (H_1, \dots, H_d)$ be a random vector
  in $\Rr^d$. For $x \in \Rr$, denote by $M(x)$ the unique element of
  $(x + 2 \pi \Zz)\cap(-\pi, \pi]$. Let $N$ be the random vector
  $(M(H_1), \dots, M(H_d))$ and let $\eta$ denote the law of $N$. Then
  $$
  \wass_1(\mu, \nu) \leqslant 	\wass_1(\mu* \eta, \nu * \eta) + 2 \expect(|H|).
  $$
\end{lemma}

\begin{proof}
  Let $u \in \mathrm{Lip}^{2\pi}_1(\Rr^d, \Rr)$. By the triangle
  inequality, we have
  \begin{align*}
    \Bigl| \int_{Q^d} u d\mu- \int_{Q^d} u d \nu\Bigr|  &\leqslant 	\Bigl| \int_{Q^d} u d\mu- \int_{Q^d} u d( \mu * \eta) \Bigr| + 	\Bigl| \int_{Q^d} u d(\mu * \eta)- \int_{Q^d} u d( \nu * \eta) \Bigr| \\ &  + 	\Bigl| \int_{Q^d} u d(\nu * \eta)- \int_{Q^d} u d \nu\Bigr|
  \end{align*}
  and thanks to Theorem~\ref{th-wass} (6), 
  $$
  \Bigl| \int_{Q^d} u d(\mu * \eta)- \int_{Q^d} u d( \nu * \eta) \Bigr| \leqslant 	\wass_1(\mu* \eta, \nu * \eta).
  $$
  Thus, it suffices to show that 
  $$
  \Bigl| \int_{Q^d} u d\mu- \int_{Q^d} u d( \mu * \eta) \Bigr| \leqslant \expect(|H|),
  $$
  since the same proof will also show the inequality for the term with
  $\nu$ and $\nu * \eta$. Now, by definition of the convolution and
  the fact that $\eta$ is a probability measure, we have
  \begin{align*}
    \Bigl| \int_{Q^d} u d (\mu * \eta) - \int_{Q^d} u d \mu\Bigr|
    & = \Bigl| \int_{Q^d} \Bigl( \int_{Q^d} u(x+y) - u(x) d \mu(x) \Bigr) d \eta(y) \Bigr| \\
    & \leqslant   \int_{Q^d} \Bigl( \int_{Q^d} |u(x+y) - u(x)| d \mu(x) \Bigr) d \eta(y).
  \end{align*}
  Using the fact that $u \in \mathrm{Lip}^{2\pi}_1(\Rr^d, \Rr)$, we
  deduce that
  $$
  \Bigl| \int_{Q^d} u d (\mu * \eta) - \int_{Q^d} u d \mu\Bigr| \leqslant \int_{Q^d} |y| d \eta(y) = \expect(|N|).
  $$
  Finally, since for all $x \in \Rr$, $|M(x)| \leqslant |x|$, we have $\expect(|N |) \leqslant \expect(|H|)$, hence the conclusion.
\end{proof}

\begin{lemma}\label{lem: choice of H}
  For all integers $T \geqslant 1$, there exists a random variable $H$
  with values in $\Rr^d$ whose characteristic function is supported on
  the cube $[-T,T]^d$ and such that
  $\expect(|H|) \leqslant \frac{2 \sqrt{3}\sqrt{d}}{T}$.
\end{lemma}

\begin{proof}
  As a first attempt (that will require some adjustments) let us use
  what is known on the Fourier transform of the triangle function: it
  is a classic fact that if $X$ is a real random variable whose
  density function is
  $$
  f(x) = \frac{2(1 - \cos(x/2))}{\pi x^2}
  $$
  then its characteristic function is given by
  $$
  \varphi_X(s) = (1 - 2 |s|)^{+}
  $$
  where $(-)^{+}$ denotes the positive part. In particular,
  $\varphi_X$ is supported on $\left[ -\frac12, \frac12 \right]$, so
  that suitable renormalizations of $X$ will allow us to construct
  random variables with support in $[-T, T]$ for all $T$. However, the
  issue is that $\expect(|X|) = + \infty$, which would make the
  inequality of Lemma \ref{lem: regularization cost} useless. This is
  why the following adjustment is needed: rather than working with
  $X$, we will work with the random variable $\xi$ whose
  characteristic function is given by
  \begin{equation} \label{eq: def v*v}
    w(s) = 3 (\varphi_X * \varphi_X) (s) = 3 \int_{\Rr} \varphi_X(s-t)\varphi_X(t) dt.
  \end{equation}
  Here, the factor $3$ is just a normalization factor to ensure that
  $w(0) = 1$, which is a necessary condition in order to be a
  characteristic function. Then thanks to the convolution theorem for
  the Fourier transform, $w$ is the characteristic function
  $\varphi_\xi$ of a random variable $\xi$ with density
  $$
  g(x) =6 \pi f(x)^2 = \frac{24 (1 - \cos(x/2))^2}{\pi x^4}\cdot
  $$
  Since $w = 3(\varphi_X * \varphi_X)$, it is supported on $[-1,1]$,
  and this time $\expect(|\xi|) < + \infty$. Even better,
  $\expect(\xi^2)$ is also finite, and can be explicitly computed!
  Indeed, since $w$ is the characteristic function of $\xi$, the
  second moment of $\xi$ is equal to $-w''(0)$, and this can be
  calculated from \eqref{eq: def v*v}, yielding $\expect(\xi^2) =
  12$. To conclude, it suffices to define $H$ as
  $\frac{1}{T}(\xi_1, \dots, \xi_d)$ for independent random variables
  $\xi_i$ having the same distribution as $\xi$. Indeed,
  $\expect(|H|^2) = \frac{12d}{T^2}$ and
  $(\expect|H|)^2 \leqslant \expect(|H|^2)$ so
  $\expect(|H|) \leqslant \frac{2 \sqrt{3}\sqrt{d}}{T}$. Moreover, for all
  $s =(s_1, \dots, s_d) \in \Rr^d$,
  $$
  \varphi_H(s) = \prod_{j = 1}^{d} w(s_j / T)
  $$
  so the choice of a $w$ supported on $[-1, 1]$ implies that
  $\varphi_H$ is supported on $[-T,T]^d$.
\end{proof}

\begin{rem}
  Among all random vectors $H$ such that $\varphi_H$ is supported on
  $[-T,T]^d$, what is the lowest $\expect(|H|)$ one can hope for? In
  this remark, we show that the result of Lemma \ref{lem: choice of H}
  is close to optimal.

  Let us denote by $(e_1, \dots, e_d)$ the canonical basis of $\Rr^d$ and by
  $$
  \begin{array}{ccccc}
    \varphi_H & : & \Rr^d & \to & \Cc \\
              && t = (t_1, \dots, t_d) & \mapsto & \expect(e^{i t \cdot H})
  \end{array}
  $$ 
  the characteristic function of $H$. Then for all $j \in \{1, \dots, d\}$,
  $$
  \frac{\partial \varphi_H}{\partial t_j}(t) = i \expect(H_j e^{i t
    \cdot H}).
  $$
  Therefore, if $\varphi_H$ is supported on $[-T,T]^d$, we have $\varphi_H(Te_j) = 0$, hence
  $$
  1 = \Bigl| \varphi_H(Te_j) - \varphi_H(0)\Bigr| = \Bigl|
  \int_{0}^{T} \frac{\partial \varphi_H}{\partial t_j}(se_j)ds \Bigr|
  \leqslant T \expect(|H_j|)
  $$
  thanks to the triangle inequality. Summing over $j$ and using Cauchy--Schwarz inequality yields
  $$d \leqslant T  \expect(|H_1| + \cdots + |H_d|) \leqslant T \sqrt{d} \expect(|H|)$$ 
  hence 
  $$
  \frac{\sqrt{d}}{T} \leqslant  \expect(|H|).
  $$
  Therefore the construction of $H$ in Lemma \ref{lem: choice of H}
  gives the best possible dependence with respect to $T$ and $d$. The
  only question that remains is whether one can obtain
  $\expect(|H|) = c \frac{\sqrt{d}}{T}$ for some
  $1 \leqslant c < 2\sqrt{3}$.
\end{rem}

\begin{proof}[Proof of Theorem $\ref{th-wass}, (7)$]
  Let $T \geqslant 1$ and let $H$ be a random vector as given by the
  previous lemma. If $\mu$ and $\nu$ are two Borel measures on
  $(Q^d, \rho_d)$, then thanks to Lemma \ref{lem: regularization cost}
  we have
  $$
  \wass_1(\mu, \nu) \leqslant 	\wass_1(\mu* \eta, \nu * \eta) + 2 \expect(|H|) 
  $$
  The choice of $H$ ensures that $\expect(|H|) \leqslant \frac{2 \sqrt{3}\sqrt{d}}{T}$, so it just remains to prove that 
  $$
  \wass_1(\mu* \eta, \nu * \eta) \leqslant \Bigl(\sum_{\substack{m \in \Zz^d \\ 0 < \|m\|_{\infty} \leqslant T}} \frac{|\hat{\mu}(m) - \hat{\nu}(m)|^2}{|m|^2}\Bigr)^{1/2}.
  $$
  This follows from Corollary \ref{lem: without smoothing} and the
  fact that $\hat{\eta}(m) = 0$ for all $m \notin [-T,T]^d$.
\end{proof}

\bibliography{wasserstein}

\end{document}